\documentclass[twoside,12pt]{amsart}
\usepackage{amssymb}
\usepackage{amscd}
\usepackage{fullpage}
\usepackage[safe]{tipa}

\numberwithin{equation}{section}

\newtheorem{Proposition}[equation]{Proposition}
\newtheorem{Lemma}[equation]{Lemma}
\newtheorem{Theorem}[equation]{Theorem}
\newtheorem{Corollary}[equation]{Corollary}

\theoremstyle{definition}
\newtheorem{Remark}[equation]{Remark}
\newtheorem{Algorithm}[equation]{Algorithm}

\def\iso{\cong}
\def\into{\hookrightarrow}
\def\onto{\twoheadrightarrow}
\def\isoto{\stackrel{\sim}{\to}}

\def\sub{\subseteq}

\def\-mod{\text{-mod}}

\def\lan{\langle}
\def\ran{\rangle}

\def\bar{\overline}

\def\ad{{\operatorname{ad}\,}}

\def\ne{{\operatorname{ne}}}
\def\reg{{\operatorname{reg}}}
\def\part{{\operatorname{part}}}
\def\Ann{\operatorname{Ann}}
\def\Pyr{\operatorname{Pyr}}
\def\RS{\operatorname{RS}}

\def\cVA{\mathcal{V\!A}}

\def\C{{\mathbb C}}

\def\Z{{\mathbb Z}}

\def\op{\operatorname{op}}
\def\Prim{\operatorname{Prim}}
\def\col{\operatorname{col}}
\def\row{\operatorname{row}}
\def\End{{\operatorname{End}}}

\def\GL{\operatorname{GL}}
\def\Sp{\operatorname{Sp}}
\def\SO{\operatorname{SO}}
\def\O{\operatorname{O}}
\def\Lie{\operatorname{Lie}}

\def\a{\mathfrak a}
\def\b{\mathfrak b}
\def\g{\mathfrak g}
\def\h{\mathfrak h}
\def\k{\mathfrak k}

\def\m{\mathfrak m}
\def\n{\mathfrak n}
\def\p{\mathfrak p}
\def\q{\mathfrak q}
\def\r{\mathfrak r}
\def\s{\mathfrak s}
\def\t{\mathfrak t}

\def\gl{\mathfrak{gl}}
\def\sl{\mathfrak{sl}}
\def\so{\mathfrak{so}}
\def\sp{\mathfrak{sp}}

\def\cO{\mathcal O}
\def\cL{\mathcal L}

\def\bp{{\mathbf p}}
\def\bq{{\mathbf q}}
\def\br{{\mathbf r}}

\def\sgn{{\operatorname{sgn}}}
\def\Tab{{\operatorname{Tab}}}
\def\sTab{{\operatorname{sTab}}}
\def\Row{{\operatorname{Row}}}

\def\coord{{\operatorname{coord}}}

\def\word{{\operatorname{word}}}
\def\part{{\operatorname{part}}}
\def\c{{\operatorname{c}}}
\def\eps{\epsilon}
\def\mf {\mathfrak}

\newdimen\Hoogte    \Hoogte=12pt    
\newdimen\Breedte   \Breedte=12pt   
\newdimen\Dikte     \Dikte=0.5pt    

\newenvironment{Young}{\begingroup
       \def\vr{\vrule height0.8\Hoogte width\Dikte depth 0.2\Hoogte}
       \def\fbox##1{\vbox{\offinterlineskip
                    \hrule height\Dikte
                    \hbox to \Breedte{\vr\hfill##1\hfill\vr}
                    \hrule height\Dikte}}
       \vbox\bgroup \offinterlineskip \tabskip=-\Dikte \lineskip=-\Dikte
            \halign\bgroup &\fbox{##\unskip}\unskip  \crcr }
       {\egroup\egroup\endgroup}

\def\Diagram#1{\relax\ifmmode\vcenter{\,\begin{Young}#1\end{Young}\,}\else%
              $\vcenter{\,\begin{Young}#1\end{Young}\,}$\fi}

\title{Finite dimensional irreducible representations of
finite $W$-algebras associated to even multiplicity nilpotent orbits in classical Lie
algebras}
\author
{Jonathan S. Brown and Simon M. Goodwin}

\address{School of Mathematics, University of Birmingham, Birmingham, B15 3LX,~UK}
\email{brownjs@for.mat.bham.ac.uk, goodwin@for.mat.bham.ac.uk}

\thanks{2010 {\em Mathematics Subject Classification}:  17B10,
81R05.}

\begin{document}

\begin{abstract}
We consider finite $W$-algebras $U(\g,e)$ associated to even multiplicity
nilpotent elements in classical Lie
algebras. We give a classification of finite dimensional irreducible
$U(\g,e)$-modules with integral central character in terms of the
highest weight theory from \cite{BGK}.  As a corollary, we obtain a
parametrization of primitive ideals of $U(\g)$ with associated
variety the closure of the adjoint orbit of $e$ and integral central
character.
\end{abstract}

\maketitle

\section{Introduction}

Let $\g$ be a reductive Lie algebra over $\C$ and let $e \in \g$ be
nilpotent. The finite $W$-algebra $U(\g,e)$ associated to the pair
$(\g,e)$ is a finitely generated algebra obtained from $U(\g)$ by a
certain quantum Hamiltonian reduction. Finite $W$-algebras were
introduced to the mathematical literature by Premet in \cite{Pr1},
though they appeared earlier in mathematical physics under a slightly different
guise,
see for example \cite{DK}. It is proved in
\cite{DDDHK} that the definition in the mathematical physics
literature via BRST cohomology agrees with Premet's definition.
A recent survey of the theory of finite
$W$-algebras is given in \cite{Lo4}.

There is a close connection between finite dimensional irreducible
representations of $U(\g,e)$ and primitive ideals of $U(\g)$
stemming from Skryabin's equivalence, \cite{Sk}. This link was
investigated further in \cite{Gi,Pr2,Pr3,Lo1,Lo2} culminating in
\cite[Theorem 1.2.2]{Lo2}, which says that there is a bijection
between the primitive ideals of $U(\g)$ whose associated variety is
the closure of the adjoint orbit of $e$, and the orbits on the
component group of the centralizer of $e$ on the isomorphism classes
of finite dimensional irreducible $U(\g,e)$-modules.  An important
tool in these results is Losev's approach to $U(\g,e)$ via Fedosov
quantization introduced in \cite{Lo1}.

Further motivation for the study of finite $W$-algebras comes from
noncommutative deformations of singularities, see \cite{Pr1};
representation theory of modular reductive Lie algebras, see
\cite{Pr1,Pr3,Pr4}; and representation theory of degenerate
cyclotomic Hecke algebras, see \cite{BK3}.

Despite the high level of recent interest, the representation theory
of finite $W$-algebras is only well-understood in certain special
cases. For $\g = \gl_n(\C)$ a thorough study of the representation
theory of $U(\g,e)$ was undertaken by Brundan and Kleshchev in
\cite{BK1,BK2}.  They obtained a classification of finite
dimensional irreducible modules along with\ numerous other results.
Several interesting consequences of this type A theory have been
found, see \cite{BB,Bru1,Bru2,BK3}.  Recent work of the first author
gives a classification of finite dimensional irreducible
$U(\g,e)$-modules for {\em rectangular nilpotent orbits} when $\g$
is of classical type, see \cite{Bro1,Bro2}.

In \cite{BGK}, a highest weight theory for representations of finite
$W$-algebras was developed.  Verma modules for $U(\g,e)$ are defined
and it is shown that any finite dimensional irreducible
$U(\g,e)$-module is isomorphic to the irreducible head of a Verma
module.  The classifications of finite dimensional irreducible
$U(\g,e)$-modules for the cases considered in \cite{BK2} and
\cite{Bro2} can be described nicely in terms of this highest weight
theory, see \cite[\S5.2]{BGK} and \cite[\S5]{Bro2}.

Let $\g$ be a classical Lie algebra.  We say a nilpotent element $e
\in \g$ is {\em even multiplicity} if all parts of the Jordan
decomposition of $e$ have even multiplicity.  In this paper we give
a classification of finite dimensional irreducible $U(\g,e)$-modules
when $e$ is even multiplicity in terms of the highest weight theory
from \cite{BGK}, see Theorem~\ref{T:main}.  Next we introduce some
notation required to state Theorem~\ref{T:submain}, which is Theorem
\ref{T:main} in the special case that all parts of the Jordan
decomposition of $e$ have the
same parity.

Let $\tilde G = \Sp_{2n}(\C)$ (or $\O_{2n}(\C)$), let $G = \tilde G
^ \circ$, and let $\g = \Lie G$, so $\g = \sp_{2n} = \sp_{2n}(\C)$
(or $\so_{2n} = \so_{2n}(\C)$). Let $V \iso \C^{2n}$ denote the
natural $\g$-module with standard basis $e_{-n},\dots,e_{-1},
e_1,\dots,e_n$, and $\g$-invariant skew-symmetric (or symmetric)
bilinear form $(.\,,.)$ defined by $(e_i, e_j) = (e_{-i},e_{-j}) =
0$ and $(e_i, e_{-j}) = \delta_{i,j}$ for $1 \leq i,j \leq n$.  Let
$\{ e_{i,j} \mid i,j = \pm 1,\dots,\pm n \}$ be the standard basis
of  $\gl_{2n} = \gl_{2n}(\C) \iso \gl(V)$.  Now $\g$ is spanned by
elements of the form $f_{i,j} = e_{i,j} + \eta_{i,j} e_{-j,-i}$ for
$i,j = \pm 1, \dots \pm n$, where $\eta_{i,j} = \sgn(i) \sgn(j)$ if
$\mf{g} = \mf{sp}_{2n}(\C)$ and $\eta_{i,j} = 1$ if $\mf{g} =
\mf{so}_{2n}(\C)$. Let $\t = \lan f_{i,i} \mid i = 1,\dots,n \ran$
be the standard Cartan subalgebra of $\g$, and let $\Phi$ and $W$ be
the root system and Weyl group of $\g$ with respect to $\t$.  We
also let $\b$ be the Borel subalgebra of $\g$ of upper triangular
matrices in $\g$, so $\b$ contains $\t$.  Let $\Phi^+ \sub \Phi$ be
the corresponding set of positive roots. We write $\t_\Z^* \sub
\t^*$ for the integral weight lattice.

Let $\bp = (p_1 \ge \dots \ge p_{2r})$ be a partition of $2n$ such
that $p_{2i-1} = p_{2i}$ for $i = 1,\dots,r$.  The {\em symmetric pyramid}
associated to $\bp$ is a diagram in the plane consisting
of $2n$ boxes of size $2 \times 2$ positioned centrally symmetric
around the origin. Symmetric pyramids were defined in \cite{EK} for all
partitions corresponding to nilpotent elements in $\g$, see also
\cite{BruG}. There are $p_1$ boxes in the middle two rows of the
symmetric pyramid associated to  $\bp$,
then $p_3$ boxes in the next two rows out, and so on.
We define the {\em coordinate pyramid} $\coord(\bp)$ by filling the boxes of
the symmetric pyramid associated to $\bp$
with the integers
$-n,\dots,-1,1,\dots,n$ from left to right and from top to bottom.
For example, for $\bp
= (4,4,2,2)$ we have
\[
\coord(\bp) =
\begin{array}{c}
\begin{picture}(80,80) \put(20,0){\line(1,0){40}}
\put(00,20){\line(1,0){80}} \put(0,40){\line(1,0){80}}
\put(00,60){\line(1,0){80}} \put(20,80){\line(1,0){40}}
\put(20,0){\line(0,1){80}} \put(40,0){\line(0,1){80}}
\put(60,0){\line(0,1){80}} \put(0,20){\line(0,1){40}}
\put(80,20){\line(0,1){40}} \put(28,70){\makebox(0,0){{-6}}}
\put(48,70){\makebox(0,0){{-5}}} \put(8,50){\makebox(0,0){{-4}}}
\put(28,50){\makebox(0,0){{-3}}} \put(48,50){\makebox(0,0){{-2}}}
\put(68,50){\makebox(0,0){{-1}}} \put(10,30){\makebox(0,0){{1}}}
\put(30,30){\makebox(0,0){{2}}} \put(50,30){\makebox(0,0){{3}}}
\put(70,30){\makebox(0,0){{4}}} \put(30,10){\makebox(0,0){{5}}}
\put(50,10){\makebox(0,0){{6}}}
\put(40,40){\circle*{3}}
\end{picture}
\end{array}
.
\]

>From $\coord(\bp)$ we may define the nilpotent element $e = \sum
f_{i,j}$, where the sum is over positive $i$ for which the box
containing $i$ is the left neighbour of the box containing $j$. Then
the Jordan type of $e$ is $\bp$, so $e$ is an even multiplicity
nilpotent. In the above example, $e = f_{1,2} + f_{2,3} + f_{3,4} +
f_{5,6}$.

{\em For the remainder of the introduction we assume
that all parts of $\bp$ have the same parity,
though this condition is not necessary for Theorem~\ref{T:main}.}

Given $i \in \{\pm 1, \dots, \pm n\}$, let $\col(i)$ be the column
of $i$, i.e the $x$-coordinate of the centre of the box labelled by
$i$.
We define
$\row(i)$ analogously (however we use a different meaning
for $\row(i)$ in Sections \ref{S:comb} and \ref{S:hwclass}). We
define
$$
\p = \lan f_{i,j} \mid \col(i) \le \col(j) \ran, \qquad \m = \lan
f_{i,j} \mid \col(i) > \col(j) \ran,
$$
$$
\h = \lan f_{i,j} \mid \col(i) = \col(j) \ran, \qquad \g_0 = \lan
f_{i,j} \mid \row(i) = \row(j) \ran.
$$
Then $\p$  is a parabolic subalgebra of $\g$ with Levi subalgebra
$\h$, and $\m$ is the nilradical of the opposite parabolic to $\p$.
Also $\g_0$ is a minimal Levi subalgebra of $\g$ containing $e$, and
$e$ is a regular nilpotent element of $\g_0$.
We write $\Phi_0$ and $W_0$ for the root system and Weyl group of
$\g_0$ with respect to $\t$, and let $\Phi_0^+ = \Phi_0 \cap
\Phi^+$.

We define $\m_\chi = \{x - \chi(x) \mid x \in \m\} \subseteq
U(\g)$, where $\chi \in \g^*$ is dual to $e$ via the trace form, and
let $Q_\chi$ be the $U(\g)$-module $U(\g)/U(\g)\m_\chi$. We note
that $Q_\chi$ is isomorphic to the induced module $U(\g)
\otimes_{U(\m)} \C_\chi$, where $\C_\chi$ is the $1$-dimensional
$U(\m)$-module given by $\chi$. The {\em finite $W$-algebra}
associated to $e$ is defined to be
$$
U(\g,e) = \End_{U(\g)}(Q_\chi)^{\op}.
$$
By the PBW theorem we have $Q_\chi \iso U(\p)$ as vector spaces,
then by a Frobenius reciprocity argument $U(\g,e)$ is isomorphic to
the subalgebra of $U(\p)$ of twisted $\m$-invariants:
$$
U(\g,e) \iso \{u \in U(\p) \mid [x,u] \in U(\g)\m_\chi \text{ for all } x \in
\m\}.
$$

By \cite[Theorem 4.3 and Lemma 5.1]{BGK}, there is a certain
subquotient of $U(\g,e)$ isomorphic to $S(\t)^{W_0}$, see also Theorem~\ref{T:Levi} and \S\ref{ss:stanlevi}.
Verma modules for $U(\g,e)$ are obtained by ``inducing''
irreducible $S(\t)^{W_0}$-modules, see \cite[\S4.2]{BGK} or \S\ref{ss:recaphw}. The finite dimensional
irreducible $S(\t)^{W_0}$-modules are given by the set $\t^*/W_0$ of $W_0$-orbits
in $\t^*$.  Given
$\Lambda \in \t^*/W_0$, we write $M(\Lambda)$ for the Verma module
corresponding to $\Lambda$.  By \cite[Theorem 4.5]{BGK}, $M(\Lambda)$
has an irreducible head, denoted $L(\Lambda)$, and any finite
dimensional irreducible $U(\g,e)$-module is isomorphic to
$L(\Lambda)$ for some $\Lambda \in \t^*/W_0$.  We note that in our
labelling we have built in ``shifts'', which we do not mention here.
These shifts are given in Section~\ref{S:levi}, where we review
highest weight theory.

We use
$\Pyr(\bp)$
to denote the set of skew-symmetric fillings of the boxes
of the symmetric pyramid associated to $\bp$ by elements of $\C$.
We identify $A \in \Pyr(\bp)$ with a weight
$\lambda_A \in \t^*$
by setting
$\lambda_A = \sum_{i=1}^n a_i \eps_i$,
where $a_i$ fills the box in $A$ occupied by
$i$ in $\coord(\bp)$
and
$\epsilon_i
= f_{-i,-i}^*$.
For example if
$$
A=
\begin{array}{c}
\begin{picture}(80,80) \put(20,0){\line(1,0){40}}
\put(00,20){\line(1,0){80}} \put(0,40){\line(1,0){80}}
\put(00,60){\line(1,0){80}} \put(20,80){\line(1,0){40}}
\put(20,0){\line(0,1){80}} \put(40,0){\line(0,1){80}}
\put(60,0){\line(0,1){80}} \put(0,20){\line(0,1){40}}
\put(80,20){\line(0,1){40}} \put(30,70){\makebox(0,0){{2}}}
\put(50,70){\makebox(0,0){{4}}} \put(8,50){\makebox(0,0){{-5}}}
\put(28,50){\makebox(0,0){{-1}}} \put(50,50){\makebox(0,0){{3}}}
\put(70,50){\makebox(0,0){{6}}} \put(8,30){\makebox(0,0){{-6}}}
\put(28,30){\makebox(0,0){{-3}}} \put(50,30){\makebox(0,0){{1}}}
\put(70,30){\makebox(0,0){{5}}} \put(28,10){\makebox(0,0){{-4}}}
\put(48,10){\makebox(0,0){{-2}}}
\put(40,40){\circle*{3}}
\end{picture}
\end{array},
$$
then
$\lambda_A = 2\epsilon_6 + 4\epsilon_5
-5\epsilon_4-\epsilon_3+3\epsilon_2+6\epsilon_1$.

The $W_0$-orbit $\Lambda_A$ of $\lambda_A$ identifies with the row
equivalence class of $A$; we denote this row
equivalence class by $\bar A$.  We restrict attention to
$\{A \in \Pyr(\bp) \mid \lambda_A \in \t_\Z^* \}$ in this paper, which amounts to the
coefficients of $\lambda_A$ all lying in $\Z$ if $\g = \sp_{2n}$,
or either all lying in $\Z$ or all lying in $\Z + \frac{1}{2}$ if
$\g = \so_{2n}$.

We recall that the restriction of the projection of $U(\g)$ onto
$Q_\chi$ restricts to an isomorphism from the centre $Z(\g)$ of
$U(\g)$ onto the centre of $U(\g,e)$, see the footnote to
\cite[Question 5.1]{Pr2}.  This allows us to view central characters
of $U(\g,e)$-modules as homomorphisms $Z(\g) \to \C$, as explained at
the end of \S\ref{ss:basprop}. The Harish Chandra homomorphism gives an
isomorphism $Z(\g) \isoto S(\t)^W$ and this meshes well with the
subquotient of $U(\g,e)$ isomorphic to $S(\t)^{W_0}$, see
\cite[Theorem 4.7]{BGK} or \eqref{e:centre}. In particular, this means that $L(\Lambda)$
has the same central character as $L(\Lambda')$ if and only if
$\Lambda$ and $\Lambda'$ are contained in the same $W$-orbit in
$\t^*$. In terms of the pyramids this translates to the multisets of
entries of $A$ and $A'$
being equal,
where $A, A' \in \Pyr(\bp)$ are such that $\Lambda_A = \Lambda$ and $\Lambda_{A'} = \Lambda'$.

The last ingredient needed for the statement of Theorem~\ref{T:main}
is the component group $\tilde C$ of the centralizer of $e$ in $\tilde G$.  It is
well-known that $\tilde C \iso \tilde H^e/(\tilde H^e)^\circ$, where $\tilde H$ is the subgroup
of $\tilde G$ corresponding to $\h$ and $\tilde H^e$ is the centralizer of $e$ in
$\tilde H$. One can see that the adjoint action of $\tilde H$ on $\g$ induces an
action of $\tilde H$ on $U(\g,e)$. As explained in the introduction to
\cite{Lo2}, this induces an action of $\tilde C$ on the isomorphism classes
of finite dimensional irreducible $U(\g,e)$-modules.

In \S\ref{ss:comp} we define an action of $\tilde C$ on a certain subset
of
$\Pyr^\leq(\bp)$, where $\Pyr^\leq(\bp)$ is the subset
of $\Pyr(\bp)$ consisting of tables with weakly increasing rows.
To be more specific let
$\Pyr^\c(\bp)$ denote the subset of $\Pyr^\leq(\bp)$
consisting of elements which correspond to integral weights
and are row equivalent to column strict.
By column strict we mean that
all of the columns are strictly decreasing.
In type D we also call
elements of $\Pyr(\bp)$
column
strict
if all of their columns are strictly decreasing, or
if their columns are strictly decreasing everywhere, except the two middle
boxes in the middle column (if it exists)
contain 0.
Now the set on which the $\tilde C$-action is defined
is $\tilde C \cdot \Pyr^\c(\bp)$.
In the future work \cite{BroG}
it will be shown that this action corresponds
to the action of $\tilde C$ on the isomorphism classes of finite
dimensional irreducible $U(\g,e)$-modules.

We are now in a position to state the main theorem of this paper in
the case that all parts of $\bp$ have the same parity.  The more
general statement is given in Theorem~\ref{T:main}.

\begin{Theorem} \label{T:submain}
Let $\g = \sp_{2n}$ or $\so_{2n}$, let $\bp$ be as above, let $e \in \g$ be an even
multiplicity nilpotent such that the Jordan type of $e$ is $\bp$,
and let $A \in \Pyr^\leq(\bp)$ be such that
$\Lambda_A \in \t_\Z^*/W_0$. Then the
irreducible $U(\g,e)$-module $L(\Lambda_A)$ is finite dimensional if
and only if there exists
$c \in \tilde C$
such that $c \cdot A \in \Pyr^\c(\bp)$.
\end{Theorem}

We remark here that the restriction to integral weights is necessary
for the theorem to hold, see Remark~\ref{R:wrong}.  We hope to
address the non-integral case in future work.

Through the correspondence of finite dimensional irreducible
$U(\g,e)$-modules and primitive ideals of $U(\g)$ with associated
variety $\bar {G \cdot e}$ discussed above, we obtain
the following corollary.
We limit this corollary to the type C case,
since the type D case is more complicated.
We also obtain a more general corollary from Theorem~\ref{T:main}
which includes the type D case,
see Corollary~\ref{C:prim}.
For the statement, let $\rho \in \t^*$ be the half
sum of the roots in $\Phi^+$.  Given $\lambda \in \t^*$ we write
$L(\lambda)$ for the simple highest weight $U(\g)$-module with
highest weight $\lambda - \rho$.

\begin{Corollary} \label{C:subprim}
    Let $\mf{g}=\mf{sp}_{2n}$, and let $\bp$, $e$ be as in Theorem~\ref{T:submain}.
Then
\[
\{\Ann_{U(\g)} L(\lambda_A) \mid A \in \Pyr^\c(\bp)
\}
\]
is a complete set of pairwise distinct
primitive ideals of $U(\g)$ with integral central character
and associated variety $\bar{G \cdot e}$.
\end{Corollary}

We now give an outline of the contents of this paper, and point out
the most significant results. In Section~\ref{S:prelim} we review
some basic facts about finite $W$-algebras. In Section~\ref{S:levi}
we review the highest weight theory for finite $W$-algebras
introduced in \cite{BGK} and generalize it to ``Levi subalgebras''
of finite $W$-algebras; we note that some of the results here are
also contained in \cite{Lo3}.  An important result is
Proposition~\ref{P:induct}, which gives an inductive approach to
determining finite dimensional irreducible modules for finite
$W$-algebras.  In \cite{Lo3} Losev proved \cite[Conjecture
5.2]{BGK}, but used a potentially different isomorphism at a key
technical point.  In Proposition \ref{P:sameiso}, we show that these
isomorphisms are the same, this completes the verification of
\cite[Conjecture 5.2]{BGK}.
Also in \S\ref{ss:hwcomp},
we discuss how the action of the component group $\tilde C$
interacts with highest weight theory. In Section~\ref{S:comb} we
prove a variety of combinatorial results about
 generalizations of tableaux and pyramids called tables and s-tables.
The key result is Theorem~\ref{recs}, which relates a table being row equivalent
to column strict to the output of the Robinson--Schensted algorithm applied
to the word of the table.
In Section~\ref{S:hwclass} we prove Theorem~\ref{T:main}.  The key ingredients
are Corollary~\ref{C:BGKconj}, Theorem~\ref{recs} and the algorithm
of Barbasch and Vogan for calculating the associated variety of a primitive ideal
in the universal enveloping algebra of a classical Lie algebra.

\subsection*{Acknowledgments}  This research is funded by EPSRC grant EP/G020809/1.

\section{Preliminaries} \label{S:prelim}

In this section we define the finite $W$-algebra $U(\g,e)$
associated to a nilpotent element $e$ in a reductive Lie algebra
$\g$.  Then we recall some basic properties of $U(\g,e)$ that we
require later in the paper.  The definition we give here is the
definition ``via non-linear Lie algebras'' from \cite[\S2.2]{BGK},
which is essentially the same as the Whittaker model definition
given in \cite{Pr1}, see \cite[\S2.4]{Pr2} and \cite[Theorem
2.4]{BGK}.

\subsection{Notation and definition of $U(\g,e)$} \label{ss:Wdef}

Let $\g$ be the Lie algebra of a reductive algebraic group
$\tilde G$ over $\C$, and let $G = \tilde G ^ \circ$. Let $(\cdot|\cdot)$ be a non-degenerate symmetric
invariant bilinear form on $\g$.  For $x \in \g$ and a subspace $\a$
of $\g$, we write $\a^x = \{y \in \a \mid [y,x] = 0 \}$ for the
centralizer of $x$ in $\a$; for a subgroup $A$ of $\tilde G$ we write $A^x$
for the centralizer of $x$ in $A$.

Let $\t$ be a maximal toral subalgebra of $\g$.
We write $\Phi \sub \t^*$ for the root system of $\g$ with respect
to $\t$.
The usual pairing between
$\t^*$ and $\t$ is denoted by $\lan \cdot, \cdot \ran$.
We note that $(\cdot|\cdot)$ induces a nondegenerate symmetric form on both
$\t$ and $\t^*$.
For $\alpha \in \Phi$ we write $\alpha^\vee$
for the corresponding coroot.

Let $e \in \g$ be a nilpotent element, and define the linear map
$\chi:\g \rightarrow \C$ by $\chi(x) = (e | x)$. By the
Jacobson--Morozov theorem, we can find $h, f \in \g$ so that
$(e,h,f)$ is an $\sl_2$-triple in $\g$.
Let
$$
\g = \bigoplus_{j \in \Z} \g(j)
$$
be the $\ad h$-eigenspace decomposition, i.e.\ $\g(j) = \{x \in \g
\mid [h,x] = jx\}$.

We define the following subspaces of $\g$
$$
\p = \bigoplus_{j \geq 0} \g(j),
\qquad\n = \bigoplus_{j < 0} \g(j),
\quad\m = \bigoplus_{j \leq -2} \g(j),
\qquad\h = \g(0), \qquad\k =
\g(-1).
$$
In particular, $\p$ is a parabolic subalgebra of $\g$ with Levi
factor $\h$ and $\n$ is the nilradical of the opposite parabolic. We
let $b_1,\dots,b_r$ be a homogeneous basis for $\n$ such that $b_i
    \in \g(-d_i)$ and has weight $\beta_i \in \Phi$, where $d_i \in \Z_{>0}$.

To formulate the definition of $U(\g,e)$ we use an easy special case
of the notion of a non-linear Lie superalgebra from \cite[Definition
3.1]{DK}, where the grading is concentrated in degree zero, see
\cite[\S2.2]{BGK} for the definition of nonlinear Lie algebras in this case.

We define a symplectic form $\lan \cdot|\cdot \ran$ on $\k$ by $\lan x|y \ran =
\chi([y,x])$.
Let $ \k^{\ne} = \{x^{\ne}\mid x \in \k\}$ be a ``neutral'' copy of
$\k$. We write $x^{\ne} = x(-1)^\ne$ for any element $x \in
\g$. Now make $\k^{\ne}$ into a non-linear Lie
algebra with non-linear Lie bracket defined by $ [x^\ne,y^\ne] =
\lan x|y \ran$ for $x, y \in \k$.   Note that  $U(\k^\ne)$ is
isomorphic to the Weyl algebra associated to $\k$ and the form $\lan \cdot|\cdot \ran$.
We view
$ \tilde{\g} = \g \oplus \k^{\ne} $ as a non-linear Lie
algebra with bracket obtained by extending the brackets already
defined on $\g$ and $\k^\ne$ to all of $\tilde{\g}$, and declaring
$[x, y^{\ne}] = 0$ for $x \in \g, y \in \k$. Then $U(\tilde{\g})\iso
U(\g) \otimes U(\k^{\ne})$. Also let $\tilde{\p} = \p \oplus
\k^{\ne}$; this is a subalgebra of $\tilde \g$ whose universal
enveloping algebra is identified with $U(\p) \otimes U(\k^{\ne})$.

We define $\tilde \n_\chi = \{x - x^\ne - \chi(x) \mid x \in \n\}$.
By the PBW theorem for $U(\tilde \g)$ we have a direct sum
decomposition $U(\tilde \g) = U(\tilde \p) \oplus U(\tilde \g)\tilde
\n_\chi$. We write $\Pr : U(\tilde \g) \to U(\tilde \p)$ for the
projection along this direct sum decomposition.
We define the
{\em finite $W$-algebra}
$$
U(\g,e) = U(\tilde \p)^\n = \{u \in U(\tilde \p) \mid
\Pr([x-x^{\ne},u]) = 0 \text{ for all } x \in \n\}.
$$
It is a subalgebra of $U(\tilde \p)$ by \cite[Theorem 2.4]{BGK}.

\begin{Remark}
We note that in case the grading $\g = \bigoplus_{j \in \Z} \g(j)$
is even, meaning that $\g(j) = 0$ for odd $j$, we have $\k = 0$.  So
we do not require nonlinear Lie algebras and $U(\g,e) \sub U(\p)$; this is the case
for the definition of $U(\g,e)$ given in the introduction.
\end{Remark}

\subsection{Basic properties of $U(\g,e)$} \label{ss:basprop}

A Lie algebra homomorphism
\begin{equation} \label{e:theta}
\theta : \g^e \into U(\tilde \p)
\end{equation}
is defined in \cite[Theorem 3.3]{BGK} reformulating the definition given
in \cite[\S2.5]{Pr2}.  The restriction of $\theta$ to $\h^e$ gives an inclusion $\h^e \into U(\g,e)$, allowing us to
view $\h^e$ as a subalgebra of $U(\g,e)$.  In particular this gives
an adjoint action of $\h^e$ on $U(\g,e)$ and $U(\tilde \p)$.

We can modify $\theta$ to obtain a (non-unique) $\h^e$-equivariant linear map
\begin{equation} \label{e:Theta}
\Theta: \g^e \to U(\g,e)
\end{equation}
as in \cite[Theorem 3.6]{BGK}, which is essentially a restatement of \cite[Theorem 4.6]{Pr1}.  Then for a basis $x_1,\dots,x_r$ of $\g^e$ the set of monomials
$$
\{\Theta(x_1)^{a_1}\dots\Theta(x_r)^{a_r} \mid a_1,\dots,a_r \in \Z_{\ge 0}\}
$$
forms a PBW basis of $U(\g,e)$.

We let $\tilde C =\tilde G^e/(\tilde G^e)^\circ$ be the component group of the
centralizer $\tilde G^e$ of $e$ in $\tilde G$.  For connected $G$, it is a
standard result that $C  = G^e/(G^e)^\circ \iso H^e/(H^e)^\circ$, where $H$ is the
connected subgroup of $G$ corresponding to $\h$. This can be proved
by noting that the centralizer of $G$ must normalize $\p$, see
\cite[Proposition 5.9]{Ja}, and that any two
Levi subalgebras of $\p$ are conjugate by an element of the
unipotent radical of $P$, where $P$ is the parabolic subgroup of $G$
corresponding to $\p$.  These arguments work just as well for  $\tilde G$,
so we have $\tilde C \iso \tilde H^e/(\tilde H^e)^\circ$ in general. It is
straightforward to see that the adjoint action of $\tilde H^e$ on $\g$
gives rises to an action of $\tilde H^e$ on $U(\g,e)$.

Lastly we consider the centre of $U(\g,e)$.  The
footnote to \cite[Question 5.1]{Pr2} says that the restriction of
$\Pr$ to $Z(\g)$ gives an isomorphism $Z(\g) \isoto
Z(\g,e)$, where $Z(\g,e)$ denotes the centre of $U(\g,e)$.
Let $L$ be an irreducible $U(\g,e)$-module.  Then the centre $Z(\g,e)$ of $U(\g,e)$ acts on $L$ via a character.
We say that $L$ is of central character $\psi : Z(\g) \to \C$ if $\Pr(z)v = \psi(z)v$ for all $z \in Z(\g)$ and $v \in L$.

\iftrue
\subsection{Losev's map of ideals} \label{ss:skrylos}

In \cite{Lo2} Losev shows that there exists a map $\cdot ^\dagger$
from the set of ideals of $U(\g,e)$ to the set of ideals of $U(\g)$
such that the restriction
\begin{equation} \label{eq:dagger}
I \mapsto I^\dagger : \Prim_0 U(\mf{g},e) \twoheadrightarrow
\Prim_{e}U(\mf{g})
\end{equation}
is a surjection.
Here
$\Prim_{0}U(\mf{g},e)$
denotes the primitive ideals of
$U(\mf{g},e)$ of finite co-dimension, and
 $\Prim_{e}U(\mf{g})$
denotes the primitive ideals of $U(\mf{g})$ with associated variety
equal to $\overline{G.e}$.  For a definition of associated
varieties, see for example \cite[\S9]{Ja}.

The set
$\Prim_0 U(\g,e)$
identifies naturally with the set of isomorphism
classes of finite dimensional irreducible $U(\g,e)$-modules.
The action of $H^e$ on $U(\g,e)$ induces an action on $\Prim_0(U(\g,e)$.
The action of $\h^e$ of $U(\g,e)$ obtained from differentiating
the action of $H^e$ coincides with the adjoint action of $\h^e$ through
$\theta$, see for example \cite[Theorem 3.3(i)]{BGK}.
Therefore, the action
$H^e$ on $\Prim_0 U(\g,e)$ factors through $C$, as explained in
the introduction to \cite{Lo2}.  Putting this all together we obtain
an action of $C$ on the set of isomorphism classes of finite
dimensional $U(\g,e)$-modules.  We note that this action can also
be described in terms of ``twisting'' the action of $U(\g,e)$ on its
finite dimensional irreducible modules.

In
\cite[Theorem 1.2.2]{Lo1}
and
\cite[Theorem 1.2.2]{Lo2}
the following properties of $\cdot^\dagger$ are established:

\begin{enumerate}
\item[(i)]
    The fibers of the restriction of $\cdot ^\dagger$ in \eqref{eq:dagger}
        are precisely the $C$-orbits in $\Prim_0U(\g,e)$.
\item[(ii)]
Central characters
are preserved by
$\cdot^\dagger$
in the sense that
    if $L$ is an irreducible $U(\g,e)$-module with
    central character $\psi : Z(\g) \to \C$, then
    $(\Ann_{U(\g,e)} M)^\dagger \cap Z(\g) = \ker \psi$.
\end{enumerate}

\fi

\section{Highest weight theory and ``Levi subalgebras'' of $U(\g,e)$} \label{S:levi}

In this section we review the highest weight theory for finite
$W$-algebras from \cite{BGK}. Furthermore, we extend some of the
results from {\em loc.\ cit.\ }to define certain subquotients of
$U(\g,e)$ that play the role of Levi subalgebras; they are
isomorphic to smaller finite $W$-algebras. Such subquotients were
first used to study the representation theory of finite $W$-algebras
by Losev in \cite{Lo3}. This isomorphism is recorded in Theorem
\ref{T:Levi} and is a generalization of \cite[Theorem 4.3]{BGK}.
Using Theorem~\ref{T:Levi} we set up an inductive approach to
determining the finite dimensional irreducible modules for $U(\g,e)$
as set out in Proposition~\ref{P:induct}.  A number of the results
involved are straightforward generalizations from \cite[\S4]{BGK},
some of which are contained in \cite{Lo3}. As our setup is slightly
different to that in \cite{Lo3}, we include all statements here.

Of particular importance in this section is Corollary
\ref{C:BGKconj}, which completes the verification of \cite[Conjecture 5.2]{BGK}.
This gives a combinatorial criteria for an irreducible highest weight module for $U(\g,e)$
to be finite dimensional in the case $e$ is of standard Levi type.

\subsection{Notation for highest weight theory} \label{ss:full}
It is a standard result that a Levi factor of $\g^e$ is given by
$\h^e$. We may assume that our maximal toral subalgebra $\mf{t}$ is
contained in $\h = \g(0)$ so that $\t^e$ is a maximal toral subalgebra of
$\h^e$.  We let $\g_0 = \{x \in \g \mid [t,x] = 0 \text{ for all
} t \in \t^e\}$ be the centralizer of $\t^e$ in $\g$.  Then $\g_0$
is a Levi subalgebra of $\g$ and $e$ is a distinguished nilpotent
element of $\g_0$.   We choose a Borel subalgebra $\b_0$ of $\g_0$
contained in $\p$ and containing $\t$.
We write $\Phi_0$ for the root
system of $\g_0$ with respect to $\t$ and $\Phi_0^+$ for the
positive roots determined by $\b_0$.

We choose a parabolic subalgebra $\q$ of $\g$ with Levi factor
$\g_0$.  We write $\q_u$ for the nilradical of $\q$. Then $\b = \b_0
\oplus \q_u$ is a Borel subalgebras of $\g$. The system of positive
roots determined by $\b$ is denoted by $\Phi^+$ and we let $\rho =
\frac{1}{2}\sum_{\alpha \in \Phi^+} \alpha$.

We say a subalgebra $\r$ of $\t^e$ is  a {\em full subalgebra} if
$\r$ is equal to the centre of $\g^\r = \{x \in \g \mid [t,x] = 0
\text{ for all } t \in \r\}$. For a full subalgebra $\r$ of $\t^e$
there is an adjoint action of $\r$ on $\tilde \g$, which extends the
adjoint action of $\r$ on $\g$, created by declaring that $\r$ acts
on $\k^\ne$ by $[t,x^\ne] = [t,x]^\ne$ for $t \in \r$ and $x \in
\k$.  For an $\r$-stable subspace $\a$ of $\tilde \g$, we define
$\a^\r = \{x \in \a \mid [t,x] = 0 \text{ for all } t \in \r\}$. We
note that $\g^\r$ is a reductive Levi subalgebra of $\g$, and $e,h,f
\in \g^\r$, so we can define the finite $W$-algebra $U(\g^\r,e)$ as
in \S\ref{ss:Wdef}. We also have that $\g_0 \subseteq \g^\r$, so
$\g^\r$ is an ``intermediate'' Levi subalgebra lying over $\g_0$.

Now let $\r$ and $\s$ be full subalgebras of $\t^e$ with $\r \sub \s$.  Then we can form the $\s$-weight space decomposition
$$
\g^\r = \g^\s \oplus \bigoplus_{\alpha \in \Phi^\r_\s} \g^\r_\alpha
$$
of $\g^\r$,
where $\Phi^\r_\s \sub \s^*$ and $\g^\r_\alpha = \{x \in \g^\r \mid
[s,x] = \alpha(s)x \text{ for all } s \in \s\}$.  Then $\Phi^\r_\s$ is
a restricted root system, see \cite[\S2 and \S3]{BruG} for information
on restricted root systems.
More generally, for any subspace $\a$ of $\tilde \g^\r$ stable
under the adjoint action of $\s$ we have an
$\s$-weight space
decompositions
$$
\a= \a^\s \oplus \bigoplus_{\alpha \in \Phi^\r_\s}
\a_\alpha.
$$
We let $\q^\r_\s$ be the
parabolic subalgebra of $\g^\r$ with Levi factor $\g^\s$ and
which contains the parabolic subalgebra $\q^\r$ of $\g^\r$.  As explained in \cite[\S2]{BruG},
the parabolic subalgebra $\q^\r_\s$ gives a system $\Phi^\r_{\s,+}$ of
positive restricted roots of $\Phi^\r_\s$, namely, $\Phi^\r_{\s,+} = \{\alpha \in
\Phi^\r_\s \mid \g^\r_\alpha \sub \q^\r_\s\}$. We set $\Phi^\r_{\s,-} = -
\Phi^\r_{\s,+}$.

In much of the notation introduced above and in the next section, there are superscripts
$\r$ and subscripts $\s$.
In the case $\r = 0$, we omit this superscript, so for example we write $\q_\s$ instead of $\q^0_\s$,
and $\Phi_\s$ rather than $\Phi^0_\s$.  For the case $\s = \t^e$, we omit the subscript $\s$, so for
example we write $\q^\r$ instead of $\q^\r_{\t^e}$.  We break this convention for the restricted root systems
and write $(\Phi^\r)^e$ rather than $\Phi^\r_{\t^e}$.
Finally, in case $\r = \t^e$ we replace superscript $\r$ with subscript $0$, as in $\g_0 = \g^{\t^e}$.

We give a piece of notation that is used frequently in the remainder
of this paper. Given a character $\gamma : \a \to \C$ of a Lie
algebra $\a$, we define, the {\em shift automorphism} $S_\gamma :
U(\a) \to U(\a)$ by
\begin{equation} \label{e:shift}
S_\gamma(x) = x + \gamma(x)
\end{equation}
for each $x \in \a$.

\subsection{``Levi subalgebras''} \label{ss:levi}
Let $\r$ and $\s$ be full subalgebras of $\t^e$,
with $\r \sub \s$.

The analogue $\theta^\r$ of $\theta$ from \eqref{e:theta} gives an adjoint action of $(\h^\r)^e$ on
$U(\g^\r,e)$, which restricts to $\s$.  Therefore, we have weight space decompositions
\[
U(\tilde \p^\r) =  U(\tilde \p^\r)^\s \oplus \bigoplus_{\alpha \in \Z \Phi^\r_\s \setminus \{0\}} U(\tilde
\p^\r)_\alpha
\]
and
\[
U(\g^\r,e) = U(\g^\r,e)^\s \oplus \bigoplus_{\alpha
\in \Z \Phi^\r_\s \setminus \{0\} } U(\g^\r,e)_\alpha.
\]
The zero weight space $U(\tilde \p^\r)^\s$ is a subalgebra of $U(\tilde
\p^\r)$.  Define $U(\tilde \p^\r)_{\s,\sharp}$ to be the left ideal of
$U(\tilde \p^\r)$ generated by the root spaces $\tilde \p^\r_\alpha$
for $\alpha \in \Phi^\r_{\s,+}$.
 Then, as explained in
\cite[\S4.1]{BGK}, $U(\tilde \p^\r)^\s_\sharp = U(\tilde \p^\r)^\s \cap U(\tilde \p^\r)_{\s,\sharp}$ is a two-sided ideal of $U(\tilde
\p^\r)^\s$ so $U(\tilde \p^\r)^\s = U(\tilde \p^\s) \oplus
U(\tilde \p^\r)^\s_\sharp$.  The projection $\pi^\r_\s : U(\tilde
\p^\r)^\s \onto U(\tilde \p^\s)$ along this direct sum
decomposition induces an isomorphism $U(\tilde \p^\r)^\s / U(\tilde
\p^\r)^\s_\sharp \iso U(\tilde \p^\s)$.

Similarly, $U(\g^\r,e)^\s$ is a subalgebra of $U(\g^\r,e)$ and we
define $U(\g^\r,e)_{\s,\sharp}$ to be the left ideal of $U(\g^\r,e)$
generated by the elements $\Theta^\r(x)$ for $x \in
(\g^\r)^e_\alpha$ and $\alpha \in \Phi^\r_{\s,+}$, where $\Theta^\r$
is the analogue of the map $\Theta$ given in \eqref{e:Theta}.  Then
$U(\g^\r,e)^\s_\sharp = U(\g^\r,e)^\s \cap U(\g^\r,e)_{\s,\sharp}$
is a two sided ideal of $U(\g^\r,e)^\s$ so we can form the quotient
$U(\g^\r,e)^\s/U(\g^\r,e)^\s_\sharp$. Theorem~\ref{e:barpi} below
says that this quotient is isomorphic to $U(\g^\s,e)$.

We recall that $b_1,\dots,b_r$ is  basis of $\n$ such that $b_i \in
\g(-d_i)$ and has weight $\beta_i \in \Phi$. We let $I^\r = \{i \mid
\beta_i|_\r = 0\}$ and define
$$
\gamma^\r_\s = \sum_{\substack{i \in I^\r \\
\beta_i|_\s \in \Phi^\r_{\s,-}}}\beta_i \in \t^*.
$$
The analogue of \cite[Lemma 4.1]{BGK} says that $\gamma^\r_\s$ extends
uniquely to a character of $\p^\s$.  Therefore, we can define
the shift $S_{-\gamma^\r_\s} : U(\tilde \p^\s) \to U(\tilde
\p^\s)$ by the formula given in \eqref{e:shift} for $x \in \p^\s$,
and $S_{-\gamma^\r_\s}(y^\ne) = y^\ne$ for $y \in \k^\s$. The
following theorem can be proved using the same arguments as for
\cite[Theorem 4.3]{BGK}, which deals with the case $\r = 0$ and $\s
= \t^e$.

\begin{Theorem}\label{T:Levi}
The restriction of $S_{-\gamma^\r_\s} \circ \pi^\r_\s :
U(\tilde\p^\r)^\s \twoheadrightarrow U(\tilde \p^\s)$ defines a
surjective algebra homomorphism $U(\g^\r,e)^\s \onto U(\g^\s,e)$
with  kernel $U(\g^\r,e)^\s_\sharp$. Therefore, it induces an
isomorphism
\begin{equation} \label{e:barpi}
\bar \pi^\r_\s :
U(\g^\r,e)^\s/U(\g^\r,e)^\s_\sharp \isoto U(\g^\s,e).
\end{equation}
\end{Theorem}
We note that Losev established a similar isomorphism in this setting in
\cite{Lo3}.

\subsection{Highest weight theory and ``Levi subalgebras''} \label{ss:recaphw}

In this subsection we recall some definitions and results about
highest weight theory for $U(\g,e)$ from \cite[\S4]{BGK}.  In fact,
we work in the general setting from the previous section with $\r$
and $\s$ full subalgebras of $\t^e$, whereas the case $\r = 0$ and
$\s = \t^e$ is considered in \cite{BGK}.  All the results that we
state below can be proved in exactly the same way as in \cite{BGK},
so we simply refer to results there even though we strictly mean
their analogues.

Let $V$ be a $U(\g^\s,e)$-module.  Then, as in \cite[\S4.2]{BGK}, we define
\begin{equation} \label{e:verma}
M^\r_\s(V) = M^\r_\s(V,\q^\r) = (U(\g^\r,e)/U(\g^\r,e)_{\s,\sharp})
\otimes_{U(\g^\s,e)} V,
\end{equation}
where $U(\g^\r,e)/U(\g^\r,e)_{\s,\sharp}$ is viewed as
a right $U(\g^\s,e)$-module via the isomorphism
$\bar \pi^\r_\s$ from \eqref{e:barpi}.  The formula in \eqref{e:verma} defines a functor
$M^\r_\s: U(\g^\s,e)\-mod \to U(\g^\r,e)\-mod$, which can be
viewed as an analogue of parabolic induction for modules for
reductive Lie algebras.

There is also a right adjoint functor to $M^\r_\s$, which can be viewed as an
analogue of parabolic restriction.  This functor is denoted by
$R^\r_\s : U(\g^\r,e)\-mod \to U(\g^\s,e)\-mod$ and defined by
$$
R^\r_\s(V) = \{v \in V \mid uv = 0 \text{ for all } u \in U(\g^\r,e)_{\s,\sharp}\},
$$
where the action of $U(\g^\s,e)$ is through the isomorphism in \eqref{e:barpi}.
We remark that for $v \in V \in U(\g^\r,e)\-mod$, we have $v \in R^\r_\s(V)$ if and only if
$\Theta^\r(x)v = 0$ for all $x \in (\g^\r)^e_\alpha$ with $\alpha \in \Phi^\r_{\s,+}$.
We also note that functors similar to $M_\s^\r$ and $R_\s^\r$ were
used by Losev in \cite{Lo3}.

Let $\{V_\Lambda \mid \Lambda \in \cL_\s\}$ be a parametrization of a
complete set of pairwise non-isomorphic finite dimensional
irreducible $U(\g^\s,e)$-modules.  Then for $\Lambda \in \cL_\s$, we
define the {\em parabolic Verma module} $M^\r_\s(\Lambda) = M^\r_\s(V_\Lambda)$.
By \cite[Theorem 4.5]{BGK}, $M^\r_\s(\Lambda)$ has an irreducible head
denoted $L^\r_\s(\Lambda) = L^\r_\s(\Lambda,\q)$, and any finite dimensional irreducible
$U(\g^\r,e)$-module is isomorphic to $L^\r_\s(\Lambda)$ for some $\Lambda
\in \cL_\s$. Moreover, for $\Lambda,\Lambda' \in \cL_\s$, we have that
$L^\r_\s(\Lambda) \iso L^\r_\s(\Lambda')$ if and only if $\Lambda =
\Lambda'$.  Therefore, the $L^\r_\s(\Lambda)$'s parameterized by the set
\[
\cL^{\r,+}_\s = \{\Lambda \in \cL_\s \mid L^\r_\s(\Lambda) \text{ is finite
dimensional}\}
\]
give a complete set of pairwise non-isomorphic finite dimensional
irreducible modules for $U(\g^\r,e)$.

We define
$$
\delta^\r_\s = \sum_{\substack{i \in I^\r \\
\beta_i|_\s \in \Phi_{\s,-} \\ d_i \ge 2}} \beta_i
+ {\textstyle\frac{1}{2}}\sum_{\substack{i \in I^\r\\
\beta_i|_\s \in \Phi_{\s,-}
\\
d_i =1}} \beta_i \in \t^*,
$$
and $\delta^\r = \delta^\r_{\t^e}$.  We incorporate a shift by $\delta^\r$
into the labelling of the weight spaces, so for a
$U(\g^\r,e)$-module $M$ and $\lambda \in \s^*$, we define the
$\lambda$-weight space $M_\lambda = \{m \in M \mid \theta^\r(s)m = (\lambda(s) -
\delta^\r(s))m \text{ for all } s \in \s\}$.  The justification for this shift
is given in the next paragraph, and is based on the following commutative diagram, which is a
consequence of \cite[Lemma 4.2]{BGK}:
\[
\begin{array}{c}
\begin{CD}
\t^e & @>\theta^\r>> & U(\g^\r,e)^\s\\
@VS_{-\delta^\r_\s}VV & & @VV{\bar\pi^\r_\s}V \\
\t^e & @>\theta^\s>> & U(\g^\s,e)
\end{CD}
\end{array}.
\]

The system of positive roots $\Phi_{\s,+}$ allows us to define a
dominance ordering on $\s^*$ in the usual way: for $\lambda,\mu
\in \s^*$ we say $\lambda \le \mu$ if and only if $\mu - \lambda
\in \Z_{\ge 0}\Phi_{\s,+}$.  Let $M$ be a $U(\g^\r,e)$-module.  For
$\lambda \in \s^*$, we say that $M_\lambda$ is a {\em maximal $\s$-weight space} of $M$ if $M_\mu = 0$ whenever $\mu \in \s^*$ with $\mu > \lambda$.
In this case $M_\lambda \sub R^\r_\s(M)$, so we obtain an action of $U(\g^\s,e)$ on $M_\lambda$.
The shift by
$\delta^\r$ in the labelling of the $\s$-weight spaces means that
$\theta^\s(s)v = (\lambda(s) - \delta^\s(s))v$ for all $v \in M_\lambda$,
when $M_\lambda$ is viewed as a $U(\g^\s,e)$-module.

We say
that a $U(\g^\r,e)$-module $M$ is a {\em highest $\s$-weight module} if it
is generated by a maximal $\s$-weight space $M_\lambda$ such that
$M_\lambda$ is finite dimensional and irreducible as a
$U(\g^\s,e)$-module; we say that $M$ is of type $\Lambda \in \cL_\s$ if
$M_\lambda$ isomorphic to $V_\Lambda$.
Let $M$ be a highest $\s$-weight $U(\g,e)$-module of type $\Lambda$.
Then it follows from \cite[Theorem 4.5]{BGK} that there are unique (up to scalar)
homomorphisms $M^\r_\s(V) \to M$ and $M \to L^\r_\s(V)$.

As in \cite[\S4.4]{BGK}, we define $\mathcal O^\r_\s(e;\q_\s^\r)$ to be the
category of all (finitely generated) $U(\g^\r,e)$-modules $M$ such
that:
\begin{enumerate}
\item[(i)] the action of $\s$ on $M$ is semisimple with
finite dimensional $\s$-weight spaces; and
\item[(ii)] the set $\{\lambda \in \s^* \mid M_\lambda \neq 0\}$
is contained in a finite union of sets of the form $\{\nu \in \s^*
\mid  \nu \le \mu\}$ for $\mu \in \s^*$.
\end{enumerate}
This is an analogue of a parabolic category $\cO$ for a reductive
Lie algebra.  It is easy to see that the parabolic Verma modules
$M^\r_\s(\Lambda)$ and their irreducible heads $L^\r_\s(\Lambda)$
all lie in $\cO^\r_\s(e;\q_\s^\r)$.

We finish this subsection by giving, in Proposition~\ref{P:induct},
an inductive approach to determining finite dimensional irreducible
representations of $U(\g,e)$.
For this proposition we require a transitivity property of the
parabolic induction functors $M^\r_\s$
given in the following lemma.  Recall for the statement that our notational convention
is to omit superscript $\r$ for $\r = 0$ and omit subscript $\s$ for $\s = \t^e$.

\begin{Lemma} \label{L:indcomp}
The natural multiplication map
$$
(U(\g,e)/U(\g,e)_{\s,\sharp}) \otimes_{U(\g^\s,e)}
(U(\g^\s,e)/U(\g^\s,e)_\sharp) \to U(\g,e)/U(\g,e)_\sharp
$$
gives rise to an isomorphism of functors
$$
M_\s \circ M^\s \iso M.
$$
\end{Lemma}

\begin{proof}
The argument required is straightforward, so we omit the details.  The key
point is that $\gamma =
\gamma_\s + \gamma^\s$, which means that $\bar \pi = \bar
\pi^\s \circ \bar \pi_\s$.
\end{proof}

By our labelling convention, $\cL$ denotes a parametrization of
the set of
isomorphism classes of
finite dimensional irreducible $U(\mf{g}_0,e)$-modules.
The following proposition says that $\cL^+ \sub \cL^{\s,+}$.
\begin{Proposition} \label{P:induct}
Let $\Lambda \in \cL$, and suppose that $L(\Lambda)$ is finite dimensional.
Then $L^\s(\Lambda)$ is finite dimensional and
$$
L_\s(L^\s(\Lambda)) \iso L(\Lambda).
$$
\end{Proposition}

\begin{proof}
Let $\Lambda \in \cL^+$, so $L = L(\Lambda)$ is finite dimensional.
There exists $\lambda \in (\t^e)^*$ such that $V_\Lambda = (V_\Lambda)_\lambda$.
We consider the $U(\g^\s,e)$-module $N = R_\s(L(\Lambda))$.
The shifts in the labelling of $\t^e$-weight spaces means that we
have $N_\lambda$ is a maximal $\t^e$-weight space of $N$.
Also it is clear that $N_\lambda \iso V_\Lambda$
as $U(\g_0,e)$-modules.  Therefore, there is an epimorphism from the
submodule of $N$ generated by $N_\lambda$ onto $L^\s(\Lambda)$.
Hence, $L^\s(\Lambda)$ is finite dimensional.

We see that $L_\s(L^\s(\Lambda))$ has a maximal $\t^e$-weight
space isomorphic to $V_\Lambda$.  Therefore, there is a epimorphism
$L(\Lambda) \onto L_\s(L^\s(\Lambda))$, which must be an isomorphism as both modules are irreducible.
\end{proof}

\subsection{Centre and central characters} \label{ss:centchar}

Let $\r$ and $\s$ be a full subalgebras of $\t^e$ with $\r \sub \s$.  Recall that $\Pr : U(\tilde \g) \to U(\tilde \p)$ restricts to an isomorphism $Z(\g) \isoto Z(\g,e)$.  Thus the analogues $\Pr^\r$
and $\Pr^\s$ also restrict
to isomorphisms $\Pr^\r : Z(\g^\r) \isoto Z(\g^\r,e)$ and
$\Pr^\s : Z(\g^\s) \isoto Z(\g^\s,e)$.
To compare these isomorphisms we consider certain Harish-Chandra
isomorphisms, which we require some notation to define.  We write $W^\r$ and $W^\s$
for the Weyl groups with respect to $\t$ of $\g^\r$ and $\g^\s$ respectively, and we write $\b_u$ for
the nilradical of $\b$.
Then we define
$\Psi^\r : Z(\g^\r) \isoto S(\t)^{W^\r}$ by
\begin{align*}
z &\equiv S_\rho(\Psi^\r(z)) \mod U(\g^\r)\b_u^\r,
\end{align*}
and $\Psi^\s : Z(\g^\s) \isoto S(\t)^{W^\s}$ is defined similarly.
We recall that $\rho$ is the half sum of roots in $\Phi^+$, so that
$\Psi^\r$ is not the usual Harish-Chandra isomorphism for $\g^\r$,
as the shift is by $\rho$ rather than a half sum of positive roots
for $\g^\r$. We write $\iota^\r_\s : S(\t)^{W^\r} \into
S(\t)^{W^\s}$ for the natural inclusion.  Then the analogue of
\cite[Theorem 4.7]{BGK} says that there is a unique embedding
$c^\r_\s : Z(\g^\r) \into Z(\g^\s)$ such that the following diagram
commutes:
\begin{equation} \label{e:centre}
\begin{CD}
Z(\g^\r,e) & @<\Pr^\r << & Z(\g^\r) & @>\Psi^\r >> & S(\t)^{W^\r}\\
@V{\bar \pi^\r_\s}VV & & @VV{c^\r_\s}V & & @VV\iota^\r_\s V\\
Z(\g^\s,e) & @<\Pr^\s << & Z(\g^\s) & @>\Psi^\s >> & S(\t)^{W^\s}
\end{CD}
.
\end{equation}

Central characters for $U(\g^\r,e)$-modules and $U(\g^\s,e)$-modules are
defined as in \S\ref{ss:basprop}.
Given an irreducible finite
dimensional $U(\g^\s,e)$-module $V$, Schur's lemma tells us that
$Z(\g^\s,e)$ acts diagonally on $V$.  We write $\psi_V^\s : Z(\g^\s) \to
\C$ for the corresponding central character.
Then, by \cite[Corollary 4.8]{BGK}, the central character of $L^\r_\s(V)$ is
\begin{equation} \label{e:centchar}
\psi^\r_V = \psi^\s_V \circ c^\r_\s : Z(\g^\r) \to \C.
\end{equation}

Central characters give rise to a partition
\begin{equation}
\label{e:partdom} \cL^{\r,+}_\s = \dot{\bigcup_{\psi : Z(\g^\r) \to \C}}
\cL^{\r,+}_{\s,\psi},
\end{equation}
where $\cL^{\r,+}_{\s,\psi} = \{\Lambda \in \cL^{\r,+}_\s \mid L^\r_\s(\Lambda) \text{
has central character } \psi \}$.
The following refinement of Proposition~\ref{P:induct} is immediate from \eqref{e:centchar}.

\begin{Lemma}
Let $\psi : Z(\g) \to \C$ be a character.  Then
$$
\cL^+_\psi \sub \bigcup_{\substack{\psi' : Z(\g^\s) \to \C \\ \psi' \circ c_\s = \psi}} \cL^{\s,+}_{\psi'}.
$$
\end{Lemma}

\subsection{Finite dimensional irreducible modules for standard Levi type} \label{ss:stanlevi}

In this subsection, rather than working with full subalgebras $\r$ and $\s$ of
$\t^e$ as above, we work just in the case $\r = 0$ and $\s = \t^e$.  We
recall that we write subscript $0$ instead of superscript $\t^e$, for example
$U(\g,e)_0$ rather than $U(\g,e)^{\t^e}$.

We consider the special case where $e$ is of standard Levi type,
i.e.\ $e$ is regular in $\g_0$.
Then $\k_0 = 0$ and $\p_0 = \b_0$ is a Borel
subalgebra of $\g_0$, and $\b = \b_0 \oplus \q_u$, where $\q_u$ is
the nilradical of $\q$.
We let $\tilde \b_0$ be the opposite Borel subalgebra to
$\b_0$, and set
$\tilde \b = \tilde \b_0 \oplus \q_u$, so $\tilde \b$ is another
Borel subalgebra of $\g$.
Let $\tilde \rho$ be the half sum of the positive roots corresponding
to $\tilde \b$.
A result of Kostant in \cite[\S2]{Ko} tells
us that $U(\g_0,e) \iso S(\t)^{W_0}$, where $W_0$ denotes the Weyl group of $\g_0$ with respect to $\t$.
An  explicit isomorphism
\begin{equation} \label{e:xi}
\xi_{-\tilde \rho}: U(\g_0,e) \isoto S(\t)^{W_0}
\end{equation}
is given in \cite[Lemma 5.1]{BGK}, where $\xi_{-\tilde \rho}$ is the
composition of the natural projection $U(\p_0) \to S(\t)$ with the
shift  $S_{-\tilde \rho} : S(\t) \to S(\t)$.

The finite dimensional irreducible modules for $S(\t)^{W_0}$ are
indexed by the set $\cL = \t^*/W_0$ of $W_0$-orbits in $\t^*$.
Therefore, given $\Lambda \in \t^*/W_0$, we can define an
irreducible $U(\g_0,e)$-module $V_\Lambda$ through $\xi_{-\tilde
\rho}$. Then we have the Verma module $M(\Lambda)$ with irreducible
head $L(\Lambda)$. We note that the central character of
$L(\Lambda)$ corresponds to the $W$-orbit in $\t^*$ that contains
$\Lambda$ through \eqref{e:centre}.

Two conjectures regarding highest weight theory for $U(\g,e)$ are
given in \cite[\S5.1]{BGK}.  The first is \cite[Conjecture
5.2]{BGK}, which gives a condition for $L(\Lambda)$ to be finite
dimensional.  To state this conjecture we need to give some
notation.  We write $L(\lambda)$ for the irreducible highest weight
$U(\g)$-module with highest weight $\lambda - \rho$, with respect to
the Borel subalgebra $\b$.  The adjoint $G$ orbit of $e$ is denoted
by $G \cdot e$ and its closure by $\bar{G \cdot e}$. Then
\cite[Conjecture 5.2]{BGK} says: if $\lambda \in \Lambda$ is chosen
so that $\lan \lambda , \alpha^\vee \ran \notin \Z_{> 0}$ for all
$\alpha \in \Phi_0^+$, then $L(\Lambda)$ is finite dimensional if
and only if $\cVA(\Ann_{U(\g)} L(\lambda))  = \bar{G \cdot e}$,
where $\cVA(I)$ denotes the associated variety of an ideal $I
\subset U(\mf{g})$. We recall that $\Phi_0$ denotes the root system
of $\g_0$ with respect to $\t$ and $\Phi_0^+ = \Phi_0 \cap \Phi^+$.

The second conjecture is \cite[Conjecture 5.3]{BGK} which states that
the category $\cO(e;\q)$ is equivalent to a certain category of
generalized Whittaker modules.  To define this category, we
use $\tilde \b_u$, the
nilradical of the Borel subalgebra $\tilde \b$ defined above.

So $\tilde \b_u$ is a maximal nilpotent subalgebra of
$\q$, and $\chi$ restricts to a character of $\tilde \b_u$.  We define
$\cO(\chi,\q)$ to be the category of all finitely generated
$U(\g)$-modules $M$ that are locally finite over $Z(\g)$ and
semisimple over $\t^e$, such that $x-\chi(x)$ acts locally
nilpotently on $M$ for all $x \in \tilde \b_u$.  The category obtained by
removing the condition that $\t^e$ acts semisimply has been studied
see for example \cite{MS}. In $\cO(\chi,\q)$ there are analogues of
Verma modules that are indexed by $\t^*/W_0$ and
have irreducible heads.
In addition to predicting an equivalence of categories
$\cO(e,\q) \isoto \cO(\chi,\q)$, \cite[Conjecture 5.3]{BGK} also says
that this equivalence should send $M(\Lambda)$ to the Verma module in $\cO(\chi,\q)$
corresponding to $\Lambda \in \t^*/W_0$. It is explained in
\cite[\S5.1]{BGK} that \cite[Conjecture 5.2]{BGK} is a consequence of
\cite[Conjecture 5.3]{BGK}.

In \cite[Theorem 4.1]{Lo3}, Losev proved that a more general
equivalence of categories than that predicted by \cite[Conjecture 5.3]{BGK}
holds.
In the proof of \cite[Theorem 4.1]{Lo3}, an isomorphism $\Psi :
U(\g_0,e) \isoto U(\g,e)_0/U(\g,e)_{0,\sharp}$ is used, which is
possibly different to $\bar \pi^{-1}$ from \eqref{e:barpi}.
Verma modules are defined in \cite{Lo3}
by using
the isomorphism $\Psi$;
however, it is remarked in \cite[\S4.2]{Lo4} that $\Psi^{-1}$ may be
different from $\bar \pi$, which means the labelling of Verma
modules in \cite{Lo3} may be different from that in \cite{BGK}. The
equivalence of categories proved in \cite[Theorem 4.1]{Lo3} does
send Verma modules to Verma modules and respects labels, but the
inconvenience regarding the potentially different labels of Verma
modules in \cite{Lo3} and \cite{BGK} means that we are not able to
deduce \cite[Conjecture 5.2]{BGK} immediately from \cite[Theorem
4.1]{Lo3}. The following proposition resolves this problem.

\begin{Proposition} \label{P:sameiso}
Assume that $e$ is of standard Levi type.
Let
$$
\bar \pi, \Psi^{-1} : U(\g,e)_0/U(\g,e)_{0,\sharp} \isoto U(\g_0,e)
$$
be the
isomorphisms given by \cite[Theorem 4.3]{BGK} and \cite[Theorem
4.1]{Lo3} respectively.  Then $\bar \pi = \Psi^{-1}$.
\end{Proposition}

\begin{proof}
The composition $\bar \pi \circ \Psi : U(\g_0,e) \to U(\g_0,e)$ is
an automorphism, so through the isomorphism $\xi_{-\tilde \rho}:
U(\g_0,e) \isoto S(\t)^{W_0}$ we obtain an automorphism $\sigma$ of
$S(\t)^{W_0}$.  From \eqref{e:centre}, \cite[Theorem 4.1(1)]{Lo3}
and \cite[Theorem 1.2.2(iii)]{Lo1}, we see that $\Psi^{-1}$ and
$\bar \pi$ agree on the centre $Z(\g,e)$ of $U(\g,e)$. From
\eqref{e:centre} we therefore see that $\sigma$ fixes $S(\t)^W \sub
S(\t)^{W_0}$.  Thus the comorphism of $\sigma$ is a morphism
$\sigma_* : \t^*/W_0 \to \t^*/W_0$ that induces the identity map on
$\t^*/W$.

We define $\t^*_{\reg} = \{v \in \t^* \mid \langle v, \alpha^\vee \rangle \neq 0
\text{ for all } \alpha \in \Phi\}$, and
let $W^0$ be the set of minimal length
representatives of the right cosets of $W_0$ in $W$.
Let $x \in \t^*_\reg$ and let $X$ be the $W_0$-orbit of $x$.  Then
$\sigma_*(X) \in \t^*_\reg/W_0$, so there exists unique $w_x \in
W^0$ such that $w_x x \in \sigma_*(X)$. This gives rise to a map $f:
\t^*_\reg \to W^0$ defined by $f(x) = w_x$.  It is easy to see that
$f$ is locally constant with respect to the Euclidean topology, so,
since $\t^*_\reg$ is connected, $f$ is constant, say $f(x) = w$ for
all $x \in \t^*_\reg$.

Let $v \in W_0$ and $x \in \t^*_\reg$. Then $f(x) = f(vx) = w$, so
we have $wx = v'wvx$ for some $v' \in W_0$.  Therefore, we see that
$wvw^{-1} \in W_0$, so that $w \in N_W(W_0)$.  Thus, conjugation by
$w$ gives a map $c_w : \t^*/W_0 \to \t^*/W_0$, which agrees with
$\sigma_*$ on $\t^*_\reg/W_0$.  Hence, we must have $\sigma_* = c_w$. In
turn this means that $\sigma$ is the map on $S(t)^{W_0}$ induced by
conjugation by $w$.

By \cite[Lemma 14]{BruG}, there is a natural isomorphism $N_W(W_0)/W_0
\iso N_{G^e}(\t^e)/C_{G^e}(\t^e)$, where $N_{G^e}(\t^e)$ and $C_{G^e}(\t^e)$
denote the normalizer and centralizer of $\t^e$ in $G^e$
respectively. This isomorphism is obtained by observing that $\t^e$
is stable under the action of $N_W(W_0)$ and $W_0$ fixes $\t^e$
pointwise.    Now it follows from \cite[Remark 5.5]{Lo3} that
$\sigma$ fixes $\t^e \sub S(\t)^{W_0}$.  Hence, we must have $w = 1$
and $\sigma$ is the identity map.
\end{proof}

As a corollary we state \cite[Conjecture 5.2]{BGK},
however we emphasize that \cite[Theorem 4.1]{Lo3} encapsulates the same data,
however it uses the isomorphism $\Psi$ to define the analogue
of $L(\Lambda)$, while to prove the results in this paper we need
to use the isomorphism $\bar \pi$.

\begin{Corollary} \label{C:BGKconj}
Assume that $e$ is of standard Levi type.
Let $\Lambda \in \t^*/W_0$ and let $\lambda \in \Lambda$ be such
that $\lan \lambda , \alpha^\vee \ran \notin \Z_{> 0}$ for all $\alpha \in \Phi_0^+$.  Then
$L(\Lambda)$ is finite dimensional if and only if
$\cVA(\Ann_{U(\g)} L(\lambda))  = \bar{G \cdot e}$.
\end{Corollary}

We finish this subsection by discussing Proposition~\ref{P:induct}
in the case where $e$ is of standard Levi type.  Let $\s$ be a full
subalgebra of $\t^e$.  For $\Lambda \in \cL = \t^*/W_0$, we define
$V_\Lambda$ and the Verma module $M^\s(\Lambda) = M^\s(V_\Lambda)$
for $U(\g^\s,e)$ as above using the isomorphism $\xi_{-\tilde \rho}
: U(\g_0,e) \isoto S(\t)^{W_0}$ from \eqref{e:xi}.  Note that we use
the shift by $-\tilde \rho$ rather than $- \tilde \rho_s$, where
$\tilde \rho_\s$ is the analogue of $\tilde \rho$ for  $\g^\s$; this
can be viewed as a ``shift in origin'' as $\tilde \rho - \tilde
\rho_\s$ is orthogonal to all roots in $\Phi^\s$. With this
convention Proposition~\ref{P:induct} for the standard Levi case is
tidily stated as:

\begin{Corollary} \label{C:inductsl}
Assume that $e$ is of standard Levi type.
Let $\Lambda \in \t^*/W_0$ and suppose that $L(\Lambda)$ is
finite dimensional.  Then $L^\s(\Lambda)$ is finite dimensional.
\end{Corollary}

\subsection{Component group action} \label{ss:hwcomp}

Let $\s$ be a full subalgebra of $\t^e$ and let $\tilde G^\s$ be the
centralizer of $\s$ in $\tilde G$.
As discussed at the end of \S\ref{ss:basprop}, there is an action of
the component group $\tilde C^\s = \tilde G^{\s,e}/(\tilde G^{\s,e})^\circ$ on the set of finite dimensional
irreducible $U(\g^\s,e)$-modules; here $\tilde G^{\s,e}$ denotes the centralizer of $e$ in $\tilde G^\s$.

Given $c \in \tilde C^\s$ and an irreducible $U(\g^\s,e)$-module $L$, we
write $c \cdot L$ for the irreducible $U(\g^\s,e)$-module obtained
by twisting with $c$. By definition $c \cdot L$ is equal to $L$ as a
vector space with action given by choosing $\dot c$ in $\tilde G^{\s,e}$
that lifts $c$ and setting $u \cdot v = (\dot c \cdot u)v$ for $u
\in U(\g^\s,e)$ and $v \in L$; this only depends on the choice of
$\dot c$ up to isomorphism. This gives rise to an action of $\tilde C^\s$
on $\cL^{\s,+}$: for $c \in \tilde C^\s$ and $\Lambda \in \cL^{\s,+}$ we
write $c \cdot_\s \Lambda$ for the image of $\Lambda$ under $c$; by
definition we have $c \cdot L^\s(\Lambda) \iso L^\s(c \cdot_\s
\Lambda)$. Moreover, as the action of $\tilde C^\s$ fixes $Z(\g^\s,e)$, we
get an action on $\cL^{\s,+}_{\psi}$ for each $\psi : Z(\g^\s)
\to \C$. In the case $\s = 0$ we omit the subscript in the notation
for the action of $\tilde C$ on $\cL^+$.

The inclusion $\tilde G^{\s,e} \into \tilde G^e$ induces a injective
map $\iota: \tilde C^\s \to \tilde C$, so we can view $\tilde C^\s$
as a subgroup of $\tilde C$. We briefly explain why $\iota$ is
injective.  We can also induce $\iota$ from the inclusion $\tilde
H^{\s,e} \into \tilde H^e$.  If $x \in \tilde H^{\s,e} \cap (\tilde
H^e)^\circ$, then $x$ and $S$ generate a connected Abelian subgroup
of $\tilde H^e$, where $S$ is the torus in $\tilde G$ with Lie
algebra $\s$. Therefore, there is a Borel subgroup of $D$ of
$(\tilde H^e)^\circ$ containing $x$ and $S$.  Then the centralizer
of $\s$ in $D$, denoted $D^\s$, is a Borel subgroup of $(\tilde
H^{\s,e})^\circ$ and $x \in D^\s$. Hence, $x \in (\tilde
H^{\s,e})^\circ$.

By Proposition
\ref{P:induct}, we have $\cL^+ \sub \cL^{\s,+}$.  Therefore, given
$\Lambda \in \cL^+$ and $c \in \tilde C^\s \sub \tilde C$ we can consider
$c \cdot_\s \Lambda$ and $c \cdot \Lambda$.  The following lemma says that
these two actions of $c$ on $\Lambda$ are equal.

\begin{Lemma} \label{L:cact1}
Let $c \in \tilde C^\s$ and $\Lambda \in \cL^+$.  Then $c \cdot_\s \Lambda = c \cdot \Lambda$.
\end{Lemma}

\begin{proof}
Consider $c \cdot L_\s(L^\s(\Lambda))$ and view $N = (1 +
U(\g,e)_{\s,\sharp}) \otimes L^\s(\Lambda)$ as a subspace.  Then $N$
is a maximal $\s$-weight space of $c \cdot L_\s(L^\s(\Lambda))$
because $c \in \tilde C^\s$.  Therefore, we can view $N$ as a
$U(\g^\s,e)$-module, and as such it is isomorphic to $c \cdot
L^\s(\Lambda)$, which by definition is isomorphic to  $L^\s(c
\cdot_\s \Lambda)$.

Therefore, there is a homomorphism from the submodule of
$c \cdot L_\s(L^\s(\Lambda))$ generated by $N$ to $L_\s(L^\s(c \cdot_\s \Lambda))$.
Now using the fact that $c \cdot L_\s(L^\s(\Lambda))$ is
irreducible, we see that it must be isomorphic to $L_\s(L^\s(c \cdot_\s \Lambda))$.

Using Proposition~\ref{P:induct}, we have $c \cdot L(\Lambda) \iso c \cdot L_\s(L^\s(\Lambda)) \iso L(c \cdot_\s \Lambda)$ and by definition
$c \cdot L(\Lambda) \iso L(c \cdot \Lambda)$, so we are done.
\end{proof}

\section{Combinatorics of tables} \label{S:comb}

Our proof of Theorem~\ref{T:main} depends on combinatorics of the symmetric
pyramid associated to a nilpotent element of a classical Lie algebra
as described in Section~\ref{S:hwclass}.  In this section, we
present the underlying combinatorial results that we require for the
proof.

Throughout the remainder of this article, by a partition we mean a
multiset of positive integers; we usually denote partitions by
writing them as a sequence in either increasing or decreasing order,
and often use exponential notation to denote repeated entries.  We
write $>$ for the usual dominance ordering on partitions. Given a
partition $\bp$ we write $\bp^T$ for the transpose partition.  We
recall the elementary fact that if $\bp$ and $\bq$ are partitions
of the same integer
with $\bp \ge \bq$, then $\bp^T \le \bq^T$.

\subsection{Frames and tables} \label{ss:frames}

We define a {\em frame} to be a connected arrangement of boxes in
the plane such that the boxes are aligned in rows, and so that rows are connected.
A frame is called {\em justified} if the boxes are aligned in columns and is called
{\em left justified} if its rows all start in the same column;
so a left justified frame is justified.
We say a justified frame $F$ is {\em preconvex} if given any two columns in $F$, we can slide one of them
horizontally so that it fits entirely inside the other.  We say a frame is {\em convex} if it is preconvex
and it is has connected columns.  For
example
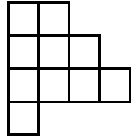
\[
\Diagram{ &  \cr  & & \cr  & & & \cr \cr }
\]
is a left justified
convex frame, and
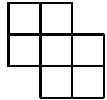
\[
\begin{picture}(36,36)
\put(12,0){\line(1,0){24}} \put(0,12){\line(1,0){36}} \put(0,24){\line(1,0){36}}
\put(0,36){\line(1,0){24}} \put(0,12){\line(0,1){24}}
\put(12,0){\line(0,1){36}} \put(24,0){\line(0,1){36}}
\put(36,0){\line(0,1){24}}
\end{picture}
\]
is a frame that is justified, but not left justified and not convex.  We note that a left
justified frame is convex if and only if its columns are connected.

A frame filled with integers is called a {\em table}. Given a table
$A$, the {\em frame of $A$} is obtained by removing the integers in
the boxes.  We say a table is justified, left justified or convex if
its frame is.

Let $F$ be a frame.  We write $\Tab(F)$ for the set of all tables
with frame $F$.  The {\em row equivalence class} of $A \in \Tab(F)$
is obtained by taking all possible permutations of the entries in
the rows of $A$; we write $\bar A$ for the row equivalence class of
$A$.
We define $\Row(F) = \{\bar A \mid A \in \Tab(F)\}$.
Given $A
\in \Tab(F)$, there is a unique element $A^\leq \in \bar A$ with
weakly increasing rows.  Let $\Tab^\leq(F) = \{A^\leq \mid A \in \Tab(F)\}$,
which is in bijection with $\Row(F)$.

Given a frame $F$ we denote by $l(F)$ the left justified frame obtained from $F$ by justifying
the rows.  Given $A$ in $\Tab(F)$ we also define $l(A) \in \Tab(l(F))$ by left justifying all of the rows of $A$.

Suppose $F$ is justified.  We say $A \in \Tab(F)$ is {\em
column strict} if the columns of $A$ are decreasing and we say that $A$ is
{\em row equivalent to column strict} if there exists column strict
$B \in \bar A$.  For example,
\[
A =  \Diagram{ 6 & 4 \cr 5  \cr  1 & 2 & 3 \cr }
\]
is column strict, but $A^\leq$ is not column strict, so $A^\leq$ is row
equivalent to column strict.  We emphasize that in our definition of
column strict we require entries in columns to be decreasing over
gaps in columns.  So for example
\[
\Diagram{ 5 & 1  \cr 4  \cr  2 & 3 \cr }
\]
is neither column strict nor row equivalent to column strict.
For a general frame $F$, we say that $A \in \Tab(F)$ is {\em justified row equivalent to
column strict} if $l(A)$ is row equivalent to column strict.  We note that
in the case $F$ is preconvex, then the notions of row equivalent to column strict
and
justified row equivalent to
column strict coincide.

Let $A \in \Tab(F)$.  We denote the number of rows of $F$ by $r_F =
r_A$ and label the rows of $F$ and $A$ with $1,\dots,r_F$ from top
to bottom. We define $\part(F) = \part(A)$ to be the partition given
by the row lengths in $F$.
For example,
if
\[
A =  \begin{array}{c} \Diagram{ 2 & 3 \cr 5  \cr 1 & 4 & 6 \cr } \end{array},
\]
then $\part(A) = (3,2,1)$.  If $\part(F) = \bp$, then we sometimes
say $F$ is associated to $\bp$.  Given $1 \le i < j \le r_A$, we
write $A_i$ for the $i$th row of $A$, and $A^i_j$ for the table
obtained from $A$ by removing rows $1,\dots,i-1$ and
$j+1,\dots,r_A$.  Note that when considering subtables of the form
$A^i_j$ we continue to use the labelling of rows inherited from $A$.
Let $\word(A)$ denote the sequence of integers created by listing
the entries in $A$ row by row from left to right, top to bottom.
With $A$ as above, we have $\word(A) = (2,3,5,1,4,6)$.

A {\em tableaux} is a left justified table $A$ with increasing row lengths  such
that $A = A^\leq$ and $A$ is column strict.  For example
\[
  \Diagram{5 & 7 \cr 2 & 4 \cr 1 & 3 & 6\cr}
\]
is a tableaux.

\subsection{The Robinson--Schensted algorithm} \label{ss:RS} Our discussion of the
Robinson--Schensted algorithm follows \cite{F}. The
Robinson--Schensted algorithm takes as input a finite list of
integers, called a {\em word}, and outputs a
tableaux.

The algorithm is defined recursively, starting with the empty
tableaux. If $w = a_1 \dots a_n$ is the word, then we assume that
$a_1, \dots, a_{i-1}$ have already been inserted. To insert $a_i$,
assume $b_1 \leq \dots \leq b_k$ is the bottom row of the tableaux.
If $a_i \geq b_k$, then insert $a_i$ at the end of the bottom row.
Otherwise there exists $j$ such that $a_i < b_j$ and $a_i \geq
b_{j-1}$. Replace $b_j$ with $a_i$, then recursively insert $b_j$
into the diagram with the bottom row removed.  In this latter
situation we say that $a_i$ {\em bumps} $b_j$. We write $\RS(w)$ for
the output of the Robinson--Schensted algorithm applied to the word
$w$. Given a frame $F$ and $A \in \Tab^\leq(F)$ we write $\RS(A) =
\RS(\word(A))$.

An important related concept is that of {\em Knuth equivalence} on
the set of words of integers. If $x, y$, and $z$ are integers for
which $x < y < z$, and $u$ and $v$ are words of integers, then we
declare that
\begin{equation} \label{knuth1}
  uyzxv \equiv uyxzv
\end{equation}
and
\begin{equation} \label{knuth2}
  uxzyv \equiv uzxyv.
\end{equation}
The equivalence relation on the set of words of integers generated
by \eqref{knuth1} and \eqref{knuth2} is the Knuth equivalence
relation. The Robinson--Schensted algorithm can be interpreted as
choosing a canonical representative in the Knuth equivalence class
of a word. It is proved in \cite[\S2]{F} that if $w$ is a word of
integers, then $w \equiv \word(\RS(w))$.

We now discuss two alternative methods for calculating $\part(\RS(w))$, which are dual to each other.
The first involves finding disjoint weakly increasing subwords of $w$.
We define $\ell(w,k)$ to be the maximum possible sum of the
lengths of $k$ disjoint weakly increasing subsequences of $w$. For
example, $\ell(412563,1) = 4$ use the increasing subword $1256$
and $\ell(412563,2) = 6$ using the subwords $456$ and $123$.  The following lemma is an
immediate corollary of \cite[Lemma 3.1.1]{F} and \cite[Lemma
3.1.2]{F}:

\begin{Lemma} \label{altRS}
Let $w$ be a word of integers and let $\bp = (p_1 \ge
\dots \ge p_n) = \part(\RS(w))$. Then for all $k \ge 1$,
$\ell(w,k) = p_1 + \dots + p_k$.
\end{Lemma}

The dual version of the above lemma considers lengths of strictly decreasing subwords.
We define \textctc$(w,k)$ to be the maximum possible sum of the lengths of
$k$ disjoint strictly decreasing subsequences of $w$.
The following lemma is a consequence of \cite[Theorem 1.6]{Gr}.

\begin{Lemma} \label{altcRS}
Let $w$ be a word of integers and let $\bp^T = (p^*_1 \ge \dots \ge
p^*_n)$ be the dual partition to $\bp = \part(\RS(w))$. Then for all
$k \ge 1$, {\em \textctc}$(w,k) = p^*_1 + \dots + p^*_k$.
\end{Lemma}

\begin{Remark} \label{R:altRS}
As the above example illustrates, if the first $k$ terms of
$\part(\RS(w))$ are $p_1, \dots, p_k$, it does not mean that we can
find $k$ disjoint weakly increasing subsequences of $w$ of lengths
$p_1, \dots, p_k$. However, one situation where this is possible is
when $\part(\RS(w))$ is of the form $(p_1^a \ge (p_1-1)^b \ge \dots )$.
In this case it
is easy to see that one must be able to find $a$ disjoint
weakly increasing subsequences of $w$
length $p_1$,
which are disjoint from
$b$ disjoint weakly increasing
subsequences of $w$ of length $p_1-1$.

The situation where we consider $\bp^T$ and strictly decreasing
subsequences is completely analogous.
\end{Remark}

\subsection{Column strict tables}

The following theorem, which is proved after three preliminary
lemmas, is the important combinatorial result required for the proof
of Theorem~\ref{T:main}.

\begin{Theorem} \label{recs}
Suppose $F$ is convex and let $A \in \Tab^\leq(F)$.  Then
$\part(\RS(A)) =
\part(F)$ if and only if $A$ is row equivalent to column strict.
\end{Theorem}

An important part of the proof of Theorem~\ref{recs} involves the
notion of row swapping.  Fix a justified preconvex frame $F$, let $A
\in \Tab^\leq(F)$ and let $1 \le i < r_F$.  We define $s_i(F)$ to be
the frame obtained from $F$ by swapping the $i$th and $(i+1)$th row;
note that $F$ being preconvex ensures that $s_i(F)$ is connected.
Below we define $s_i(A) \in \Tab^\leq(s_i(F))$ provided $A_i$ and
$A_{i+1}$ satisfy certain conditions.

Let $c_1 \le c_2 \le \dots c_s$ be the entries of $A_i$ and let $d_1
\le d_2 \le \dots d_t$ be the entries of $A_{i+1}$.  We split into
two cases.

\noindent {\bf Case 1}: $s < t$. Then $s_i(A)$ is defined if $d_i <
c_i$ for $i = 1,\dots,s$.  In this case, we choose $e_1, \dots, e_s$
from $d_1, \dots d_t$ so that $e_i < c_i$ and $\sum_{i=1}^s c_i-e_i$
is minimal. Then $e_1, \dots e_s$ form the entries of row $i+1$ in
$s_i(A)$, while the remaining entries in $A_{i+1}$ are added to
$c_1, \dots, c_s$ (and rearranged into weakly increasing order) to
form the entries of row $i$ in $s_i(A)$.

\noindent {\bf Case 2}: $s > t$.  Then $s_i(A)$ is
defined if $d_{s+1-i} < c_{t+1-1}$ for $i = 1,\dots,t$.  In this
case, we choose $e_1, \dots, e_t$ from $c_1, \dots c_s$ so that $e_i
> d_i$ and $\sum_{i=1}^t e_i -d_i$ is minimal. Then $e_1,
\dots, e_t$ form the entries of row $i$ in $s_i(A)$, while the
remaining elements from row $i$ are added to $d_1, \dots, d_t$ (and
rearranged into weakly increasing order) to form the entries of row
$i+1$ in $s_i(A)$.

For example we can apply row swapping to rows $2$ and $3$ of
\[
A = \Diagram{6 & 7 \cr 3 & 5 \cr 1 & 2 & 4\cr}
\]
and we get
\[
s_2(A) = \Diagram{6 & 7 \cr 1 & 3 & 5  \cr 2 & 4 \cr}\: .
\]

The next lemma is an easy observation about the above definition, so
we omit the proof.

\begin{Lemma} \label{L:swaptwice}
Let $A \in \Tab^\leq(F)$ and suppose that $s_i(A)$ is defined.  Then $s_i(s_i(A))$
is defined and equal to $A$.
\end{Lemma}

Our next lemma relates row swapping to the Robinson--Schensted algorithm.

\begin{Lemma} \label{swaprs}
Let $A \in \Tab^\leq(F)$,  let $1 \le i < r_A$, and suppose
$s_i(A)$ is defined.
Then $\RS(A) = \RS(s_i(A))$.
\end{Lemma}

\begin{proof}
It is easy to check that, in Case 2 above, $\RS(A^i_{i+1})$ is
precisely $s_i(A)^i_{i+1}$. It follows that $\word(A)$
is Knuth equivalent to $\word(s_i(A)))$, so that
$\RS(A) = \RS(s_i(A))$.

In order to deal with Case 1 above, we use Lemma~\ref{L:swaptwice} and apply
the result for Case 2 to $s_i(A)$.
\end{proof}

The following is the main technical lemma required for the proof of Theorem
\ref{recs}.

\begin{Lemma} \label{recs1}
Let $F$ be preconvex, and let $A \in \Tab^\leq(F)$ be such that $A$ is row equivalent to
column strict.
\begin{enumerate}
\item[(i)] Suppose also that $A^1_{i+1}$
is convex, and the length of row $i$ is greater than the length of
row $i+1$.  Then $s_i(A)$ is defined and is row
equivalent to column strict.
\item[(ii)]  Suppose also that $A^i_{r_A}$
is convex, and the length of row $i$ is less than the length of row
$i+1$.  Then $s_i(A)$ is defined and is row
equivalent to column strict.
\end{enumerate}
\end{Lemma}

\begin{proof}
We only prove part (i) as the proof of part (ii) is very similar.  We can assume $F$ is
left justified because $A$ is row equivalent to column strict if and only if
$l(A)$ is.

We first prove the lemma in the case that row $i+1$ has 1 box and
row $i$ has 2 boxes.  Let $\tilde A \in \bar A$ be column strict.
Denote the entries in column 1 of $\tilde A$ by $a_1, \dots,
a_{r(A)}$ and the entries in column 2 of $\tilde A$ above row $i$ by
$b_j, \dots, b_i$, where $1 \le j \le i$, $b_{j-1}$ does not exist,
and $a_k, b_k$ lie in row $k$. Since $A^1_{i+1}$ is convex we have
that $b_k$ is defined for all $k$ between $j$ and $i$.

As $\tilde A$ is column strict we have $a_{i+1} < a_i$, which
implies that $s_i(A)$ is defined.  Also we have that
$s_i(A)^i_{i+1}$ is $\RS(A^i_{i+1})$, so we consider
applying the Robinson--Schensted algorithm to $\word(A_i^{i+1})$.

If $a_{i+1}$ bumps $a_i$, then it is easy to see that
$s_i(A)$ is obtained from $A$ by simply moving the box
containing $b_i$ down to row $i+1$.   Then define $\tilde A' \in
\Row(s_i(A))$ to be obtained from $\tilde A$ by moving the
box containing $b_i$ down to row $i+1$.  It is clear that $\tilde
A'$ is column strict meaning that $s_i(A)$ is row
equivalent to column strict.

If $a_{i+1}$ bumps $b_i$, then we generate a new column strict table
$\hat A \in \bar A$ as follows. First swap the boxes containing
$a_i$ and $b_i$. The resulting diagram is now column strict except
$a_i$ may be larger than $b_{i-1}$. If this occurs, then also swap
the boxes containing $a_{i-1}$ and $b_{i-1}$. Again the resulting
diagram is column strict, except $a_{i-1}$ may be larger than
$b_{i-2}$. If this occurs, then swap the boxes containing $a_{i-2}$
and $b_{i-2}$. We keep repeating this process until we obtain a
column strict diagram, which must eventually happen, as $b_{j-1}$ does not
exist. Now the resulting diagram satisfies the first case, that is
the new $a_{i+1}$ bumps the new $a_i$, so $s_i(A)$ is row
equivalent to column strict.

Next we consider the general case.   Let $\tilde A \in \Row(A)$ be
column strict, and denote the entries of row $i$ of $\tilde A$ by
$c_1, \dots, c_r$ and the entries of row $i+1$ of $\tilde A$ by
$d_1, \dots, d_s$.  Suppose that when applying the
Robinson--Schensted algorithm to $\word(A^i_{i+1})$, there
exists $j,t$ such that
$1 \le
j \le s$ and $t > s$, and such that $d_j$ bumps $c_t$.
Then we form another table $\hat A \in \bar A$ as follows.
We begin by swapping the boxes containing $c_t$ and $c_j$ in $\tilde
A$.  The resulting diagram is column strict except that the entry in
the $(i-1)$th row and the $t$th column may not be greater than
$c_j$.  We can continue to swap boxes between the $j$th and $t$th
column, as in the above case to obtain $\hat A$.  Now in $\hat A$
one more of the first $s$ boxes of the $i$th row is bumped during
the application of the Robinson--Schensted algorithm to
$\word(A^i_{i+1})$ than is the case for $\tilde A$.  It follows that
we can assume, by induction, that $c_1, \dots, c_s$ are the elements
which are bumped in the Robinson--Schensted algorithm.  In this
situation, we can define $\tilde A' \in \bar{(s_i(A))}$ to
be obtained from $\tilde A$ by simply moving the boxes containing
$c_{s+1}, \dots, c_r$ down to row $i+1$.  It is clear that $\tilde
A'$ is column strict, so that $s_i(A)$ is row equivalent to
column strict.
\end{proof}

\begin{Lemma}  \label{swapdef}
Let $A \in \Tab^\leq(F)$ and $1 \le i < r_A$, and suppose that
$\part(\RS(A)) =
\part(F)$. Then $s_i(A)$ is defined.
\end{Lemma}

\begin{proof}
Let $\bp = (p_1 \ge \dots \ge p_n) = \part(F)$.  Let
$p_j$ be the length of $A_i$ and $p_k$ be the length
of $A_{i+1}$.  We assume the $p_j \ge p_k$, the other
case being entirely similar. Suppose $s_i(A)$ is
undefined.  Then it is easy to see that there is an increasing
subsequence in $\word(A^i_{i+1})$ of length greater than
$p_j$.  This implies that $\ell(\word(A),j) > p_1 +
\dots + p_j$, so that $\part(\RS(A)) > \part(F)$ by
Lemma~\ref{altRS}.
\end{proof}

We are now in a position to prove Theorem~\ref{recs}.

\begin{proof}[Proof of Theorem~\ref{recs}]
First we prove by induction on $r_A$ (the number of rows of $A$) that if $A$ is row
equivalent to column strict, then we can perform a sequence of row swaps
$s_{i_1},\dots,s_{i_m}$ so that $B = s_{i_1} \dots s_{i_m}(A)$ is defined and satisfies:
\begin{enumerate}
\item[(i)]  $B$ has increasing row lengths; and
\item[(ii)] $B$ is row equivalent to column strict.
\end{enumerate}
The case where $r_A = 0$ is trivial. If $A_1$ is a row of shortest
length, then it is clear that we can form a sequence of row swaps
satisfying the  above conditions for $A$ if and only if we can do so
for $A^2_{r_A}$ in which case we can apply induction. So we assume
that $A_{r_A}$ is of shortest length. In this case, let $k$ be
minimal such that $A_{k+1},\dots,A_{r_A}$ have the same length. Then
clearly $A^1_{k+1}$ is convex.  Therefore, we can apply
Lemma~\ref{recs} to see that $s_k(A)$ is defined and is row
equivalent to column strict. We continue by applying the row
swapping operations $s_{k-1}, \dots, s_1$ in turn.   Inductively, we
see that the table $A(j) = s_j \dots s_k(A)$ is defined and row
equivalent to column strict by Lemma~\ref{recs}. Now $A(1)_1$ is a
row of shortest length in $A(1)$, so we are now in the case above
where we can apply induction.  Hence, we can find the desired
sequence of row swaps.

Suppose that $A$ is row equivalent to column strict and let $B$ be the table
obtained from $A$ by a sequence of row swaps and satisfying (i) and (ii) above.
Then in fact $B$ must be column strict.  Now it is
a straightforward exercise to check that $\RS(B) = B$ so that
$\part(\RS(B)) = \part(F)$.  Hence using Lemma~\ref{swaprs}, we
have $\part(\RS(A)) = \part(F)$ as required.

Now assume that $\part(\RS(A)) = \part(F)$.  An inductive argument
very similar to above shows that we can perform a sequence of row swaps
$s_{i_1},\dots,s_{i_m}$ so that $B = s_{i_1} \dots s_{i_m}(A)$ is defined and satisfies:
\begin{enumerate}
\item[(i)]  $B$ has increasing row lengths; and
\item[(ii)] $\part(\RS(B)) = \part(F)$.
\end{enumerate}
The only adaptation required is to use Lemma~\ref{swapdef},
to see that the row swaps are defined and Lemma~\ref{swaprs} to see that
they preserve the output of the
Robinson--Schensted algorithm.
We can use
Lemma~\ref{altcRS} to see that $B$ must be column strict.
Now by performing all of the row swapping operations in reverse and
applying Lemma~\ref{recs1}, we see that $A$ is row equivalent to
column strict.
\end{proof}

Suppose $F$ is convex and let $\bar A \in \Row(F)$.  Below we
explain a process to determine if there exists column strict $B \in
\bar A$ and to find such $B$ if it exists.

\begin{Algorithm} \label{A:findcs}
Suppose $F$ is convex and let $\bar A \in \Row(F)$.
Let $a_1 \leq a_2 \leq \dots \leq a_m$ be all the entries of $A$.  We proceed in
steps, after the first $(i-1)$ steps we have inserted $a_1,\dots,a_{i-1}$
in $F$.

\noindent {\em $i$th step}: We consider $a_i$ and suppose that it
lies in row $j_i$.  We consider all empty boxes $b$ in row $j_i$ for
which either there is no box below $b$ or the box below $b$ has
already been filled.

\noindent {\em If no such box exists}, then we output that $A$ is not row equivalent to column strict and finish.

\noindent {\em Otherwise}, from all such $b$ we choose the one with the most boxes above it, and the rightmost one if there is more than one such $b$.  We insert $a_i$ in to this box.

\noindent After the $m$th step a column strict element of $\bar A$
is output.
\end{Algorithm}

Before we argue that this algorithm is correct, we illustrate it
with an example.  We consider
\[
A = \Diagram{6  & 9 \cr 2 & 3 & 5 & 8\cr 1 & 7 \cr 4 \cr}.
\]
By applying Algorithm~\ref{A:findcs}, we get the following sequence
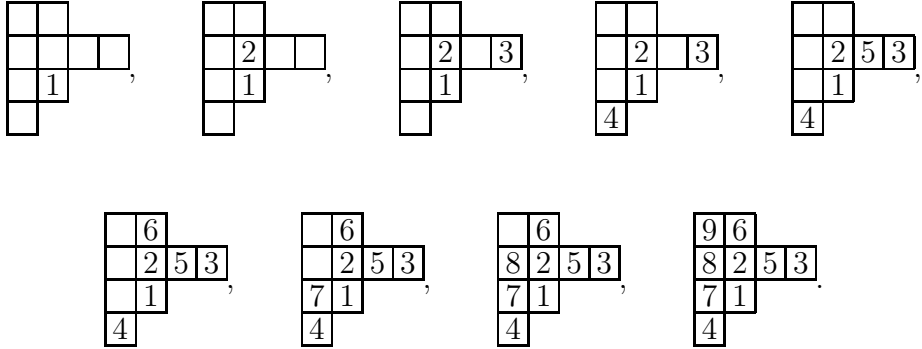
\[
\Diagram{ &    \cr  &  &  & \cr  &  1 \cr  \cr}, \qquad %
\Diagram{ &    \cr  & 2 &  & \cr  & 1 \cr  \cr}, \qquad %
\Diagram{ &    \cr  & 2 &  & 3\cr  & 1 \cr  \cr}, \qquad %
\Diagram{ &    \cr  & 2 &  & 3\cr  & 1 \cr 4 \cr}, \qquad %
\Diagram{ &    \cr  & 2 & 5 & 3 \cr  & 1 \cr 4 \cr}, \qquad %
\]
\[
\Diagram{ & 6   \cr  & 2 & 5 & 3 \cr  & 1 \cr 4 \cr}, \qquad %
\Diagram{ & 6   \cr  & 2 & 5 & 3 \cr 7 & 1 \cr 4 \cr}, \qquad %
\Diagram{ & 6   \cr 8 & 2 & 5 & 3 \cr 7 & 1 \cr 4 \cr}, \qquad %
\Diagram{ 9 & 6 \cr 8 & 2 & 5 & 3 \cr 7 & 1 \cr 4 \cr}. \qquad %
\]

\begin{proof}[Proof of correctness of Algorithm~\ref{A:findcs}]
We show by induction on $i$ that if there exists column strict $B
\in \bar A$, then there is such $B$ with boxes as filled by the
first $i$ steps of the algorithm.  The case $i =0$ is trivial.

Consider the $i$th step.  If it is not possible to find an empty
box $b$ in row $j_i$ for which either there is no box below $b$ or
the box below $b$ has already been filled, then it is clear that $A$
cannot be row equivalent to column strict.  Now suppose that $A$ is
row equivalent to column strict and that $B \in \bar A$ is column
strict.  By induction we may assume that boxes in $B$ containing
$a_1,\dots,a_{i-1}$ are filled as in the first $i-1$ steps.  Let $b$
be the box that the algorithm says to put $a_i$ in, and let $b'$ be
the box in $B$ containing $a_i$.  Now as in the proof of Lemma
\ref{recs1}, we may swap entries of $B$ in the columns containing
$b$ and $b'$ to obtain column strict $B' \in \bar A$ with $a_i$ in
$b$.  This completes the induction.
\end{proof}

\subsection{s-frames and s-tables} \label{ss:stable}

To manage the data associated to nilpotent elements in classical Lie
algebras in the next section we require a symmetric version of
frames and tables. We define an {\em s-frame} to be a frame where
the boxes, are arranged symmetrically around a central point. For
this paper, we only consider s-frames with an even number of rows.
We say that an s-frame is a {\em symmetric pyramid} if the row
lengths increase from the centre outwards; we note that a symmetric
pyramid is uniquely determined by its row lengths.

An example of an s-frame (which is not a symmetric pyramid) is
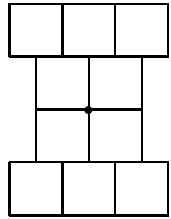
\begin{equation*}
    \begin{array}{c}
\begin{picture}(60,80)
\put(0,0){\line(1,0){60}} \put(0,20){\line(1,0){60}}
\put(10,40){\line(1,0){40}} \put(0,60){\line(1,0){60}}
\put(0,80){\line(1,0){60}} \put(0,0){\line(0,1){20}}
\put(20,0){\line(0,1){20}} \put(40,0){\line(0,1){20}}
\put(60,0){\line(0,1){20}}
\put(10,20){\line(0,1){40}} \put(30,20){\line(0,1){40}}
\put(50,20){\line(0,1){40}}
\put(0,60){\line(0,1){20}} \put(20,60){\line(0,1){20}}
\put(40,60){\line(0,1){20}} \put(60,60){\line(0,1){20}}
\put(30,40){\circle*{3}}
\end{picture}
\end{array}.
\end{equation*}

We define an {\em s-table} to be an s-frame filled with integers
skew-symmetrically with respect to the centre. Given an $s$-frame
$F$, we write $\sTab(F)$ for the set of s-tables with frame $F$. We
write $\bar A^s = \bar A \cap \sTab(A)$ for the set of s-tables row
equivalent to $A$. The subset of $\sTab(F)$ consisting of s-tables
with entries weakly increasing along rows is denoted by
$\sTab^\leq(F)$. For example

\begin{equation} \label{exstable}
\begin{picture}(80,80)
\put(20,0){\line(1,0){40}} \put(0,20){\line(1,0){80}}
\put(0,40){\line(1,0){80}} \put(0,60){\line(1,0){80}}
\put(20,80){\line(1,0){40}} \put(20,0){\line(0,1){20}}
\put(40,0){\line(0,1){20}} \put(60,0){\line(0,1){20}}
\put(60,20){\line(0,1){20}} \put(80,20){\line(0,1){20}}
\put(60,40){\line(0,1){20}} \put(80,40){\line(0,1){20}}
\put(20,80){\line(0,-1){20}} \put(40,80){\line(0,-1){20}}
\put(60,80){\line(0,-1){20}} \put(0,20){\line(0,1){40}}
\put(20,20){\line(0,1){40}} \put(40,20){\line(0,1){40}}
\put(30,70){\makebox(0,0){{-7}}}
\put(50,70){\makebox(0,0){{3}}} \put(10,50){\makebox(0,0){{-8}}}
\put(30,50){\makebox(0,0){{-4}}} \put(50,50){\makebox(0,0){{2}}}
\put(70,50){\makebox(0,0){{5}}} \put(8,30){\makebox(0,0){{-5}}}
\put(28,30){\makebox(0,0){{-2}}} \put(48,30){\makebox(0,0){{4}}}
\put(68,30){\makebox(0,0){{8}}} \put(28,10){\makebox(0,0){{-3}}}
\put(48,10){\makebox(0,0){{7}}}
\put(40,40){\circle*{3}}
\end{picture}
\end{equation}
is an s-table, which lies in $\sTab^\leq(F)$, where $F$ is its s-frame.

Let $F$ be an s-frame and $A \in \sTab(F)$.  By assumption, $F$ has
an even number of rows, say $2r$. We label the rows of $F$ and $A$
with $-r, \dots, -1,1, \dots, r$ from top to bottom. Given $i = \pm
1,\dots,\pm r$ we write $A_i$ for row of $A$ labelled by $i$, and
for $i > 0$ we write $A^{-i}_i$ for the s-table obtained by removing
rows $\pm 1,\dots, \pm (i-1)$. The table obtained from $A$ by
removing all boxes below the central point is denoted by $A^+$.  For
example if $A$ is the table above, then
\begin{equation*}
    A^+ =
    \begin{array}{c}
\begin{picture}(80,40)
\put(0,0){\line(0,1){20}} \put(20,0){\line(0,1){40}}
\put(40,0){\line(0,1){40}} \put(60,0){\line(0,1){40}}
\put(80,0){\line(0,1){20}}
\put(0,0){\line(1,0){80}}
\put(0,20){\line(1,0){80}} \put(20,40){\line(1,0){40}}
\put(30,30){\makebox(0,0){{-7}}}
\put(50,30){\makebox(0,0){{3}}} \put(10,10){\makebox(0,0){{-8}}}
\put(30,10){\makebox(0,0){{-4}}} \put(50,10){\makebox(0,0){{2}}}
\put(70,10){\makebox(0,0){{5}}}
\end{picture}
\end{array}
.
\end{equation*}

In case the row lengths of $A$ all have the same parity, then $A$ is
justified, thus there is a natural notion of $A$ being row
equivalent to column strict as an s-table. In the next lemma we see
that in fact this is not a stronger requirement than being row
equivalent to column strict.

\begin{Lemma} \label{L:srecs}
Let $F$ be an s-frame such all row lengths in $F$ have the same
parity and $A \in \sTab(F)$. Then there exists column strict $B \in
\bar A^s$ if and only if $A$ is row equivalent to column strict.
\end{Lemma}

\begin{proof}
The only if part is trivial.  Suppose that $A$ is row equivalent to column strict,
then we explain how we can adapt Algorithm~\ref{A:findcs} to find $B \in \bar A^s$,
which is column strict.

We write $a_1 >\dots > a_m$ for the positive entries in $A$.  Then in the
$i$th step of our adaptation, we insert the entries $\pm a_i$.  To insert $-a_i$
we follow the rules in Algorithm~\ref{A:findcs} and to insert $a_i$ we look for
the empty box which is leftmost such that there is no box above it or the box above it has already
been filled, and has as many empty boxes below it as possible.
(So we are mixing Algorithm~\ref{A:findcs} with its ``dual'' version.)

The proof of correctness of Algorithm~\ref{A:findcs} can be easily modified
to show that this adapted version does insert entries in to all the boxes in $F$.
It is clear that the resulting table $B$ is an s-table.  So $B \in \bar A^s$
is column strict, as required.
\end{proof}

Last in this subsection we generalize the row swapping procedure to
s-tables. As above let $F$ be an s-frame with $2r$ rows, and let $A \in
\Tab^\leq(F)$. Let $i = 1,\dots,r-1$.  We can define the row
swapping operation $s_i$ as before, so that it swaps rows $i$ and
$i+1$. Using the same rules we can define the row swapping operation
$s_{-i}$ that swaps rows $-(i+1)$ and $-i$. Now we define the
operator $\bar s_i$, on $\sTab^\leq(F)$, to be the composition of
$s_i$ and $s_{-i}$.  We note that $s_i(A)$ is defined if and only if
$s_{-i}(A)$ is defined, and that the operators $s_i$ and $s_{-i}$
commute. If $s_i(A)$ is undefined, then we say $\bar s_i$ is
undefined on $A$.  Also we note that when $s_i$ is defined, then the
action of  $s_{-i}$ is ``dual'' to that of $s_i$, so $\bar s_i(A)$
is an s-table. An example of a row swapping operation $\bar s_i$ is
\begin{equation} \label{ex1}
\bar s_1 \cdot
    \begin{array}{c}
\begin{picture}(80,80)
\put(20,0){\line(1,0){40}} \put(0,20){\line(1,0){80}}
\put(0,40){\line(1,0){80}} \put(0,60){\line(1,0){80}}
\put(20,80){\line(1,0){40}} \put(20,0){\line(0,1){20}}
\put(40,0){\line(0,1){20}} \put(60,0){\line(0,1){20}}
\put(60,20){\line(0,1){20}} \put(80,20){\line(0,1){20}}
\put(60,40){\line(0,1){20}} \put(80,40){\line(0,1){20}}
\put(20,80){\line(0,-1){20}} \put(40,80){\line(0,-1){20}}
\put(60,80){\line(0,-1){20}} \put(0,20){\line(0,1){40}}
\put(20,20){\line(0,1){40}} \put(40,20){\line(0,1){40}}
\put(30,70){\makebox(0,0){{2}}} \put(50,70){\makebox(0,0){{5}}}
\put(10,50){\makebox(0,0){{1}}} \put(30,50){\makebox(0,0){{3}}}
\put(50,50){\makebox(0,0){{4}}} \put(70,50){\makebox(0,0){{6}}}
\put(8,30){\makebox(0,0){{-6}}} \put(28,30){\makebox(0,0){{-4}}}
\put(48,30){\makebox(0,0){{-3}}} \put(68,30){\makebox(0,0){{-1}}}
\put(28,10){\makebox(0,0){{-5}}} \put(48,10){\makebox(0,0){{-2}}}
\put(40,40){\circle*{3}}
\end{picture}
\end{array}
=
\begin{array}{c}
\begin{picture}(80,80)
\put(0,0){\line(1,0){80}}
\put(0,20){\line(1,0){80}} \put(20,40){\line(1,0){40}}
\put(0,60){\line(1,0){80}} \put(0,80){\line(1,0){80}}
\put(20,0){\line(0,1){20}} \put(40,0){\line(0,1){20}}
\put(60,0){\line(0,1){20}} \put(60,20){\line(0,1){20}}
\put(80,0){\line(0,1){20}} \put(60,40){\line(0,1){20}}
\put(80,60){\line(0,1){20}} \put(20,80){\line(0,-1){20}}
\put(40,80){\line(0,-1){20}} \put(60,80){\line(0,-1){20}}
\put(0,0){\line(0,1){20}} \put(0,60){\line(0,1){20}}
\put(20,20){\line(0,1){40}} \put(40,20){\line(0,1){40}}
\put(10,70){\makebox(0,0){{2}}} \put(30,70){\makebox(0,0){{3}}}
\put(50,70){\makebox(0,0){{5}}} \put(70,70){\makebox(0,0){{6}}}
\put(30,50){\makebox(0,0){{1}}} \put(50,50){\makebox(0,0){{4}}}
\put(28,30){\makebox(0,0){{-4}}} \put(48,30){\makebox(0,0){{-1}}}
\put(8,10){\makebox(0,0){{-6}}} \put(28,10){\makebox(0,0){{-5}}}
\put(48,10){\makebox(0,0){{-3}}} \put(68,10){\makebox(0,0){{-2}}}
\put(40,40){\circle*{3}}
\end{picture}
\end{array}
.
\end{equation}

\begin{Remark} \label{R:half}
When considering orthogonal Lie algebras in Section~\ref{S:hwclass}, we
need to consider tables where the entries are elements of $\Z + \frac{1}{2}$.
Given a s-frame $F$, we define $\sTab_{\frac{1}{2}}(F)$ to be the set of all
skew-symmetric
fillings of
$F$ by elements of $\Z + \frac{1}{2}$.  We also define $\sTab_+(F) = \sTab(F) \cup \sTab_{\frac{1}{2}}(F)$ and
$\sTab_-(F) = \sTab(F)$.  All of the definitions above
also make sense for tables in $\sTab_{\frac{1}{2}}(F)$.
\end{Remark}

\section{Highest weight theory for even multiplicity nilpotent elements}
\label{S:hwclass}

The aim of this section is to prove Theorem~\ref{T:main}.
One important tool is the algorithm of Barbasch and
Vogan to determine the associated variety of the annihilator of a
highest weight module in a classical Lie algebra; this is discussed
in \S\ref{ss:BV},  and we give a slight modification in
Algorithm~\ref{A:BV2}. This algorithm is used in conjunction with
Corollary~\ref{C:BGKconj} and Theorem \ref{recs}
to help prove
Theorem~\ref{T:main}.

\subsection{Notation} \label{ss:hwnotn}

We recap some of the notation given in the introduction and give
explicit choices for the notation from \S\ref{ss:Wdef}.

We fix a positive integer $n$ and a sign $\phi \in \{\pm\}$.  Let $V
= \C^{2n}$ be the $2n$-dimensional vector space with standard basis
$\{e_{-n},\dots,e_{-1},e_1,\dots,e_n\}$ and nondegenerate bilinear
form $(\cdot,\cdot)$ defined by $(e_i,e_j) = 0$ if $i$ and $j$ have
the same sign, and $(e_i,e_{-j}) = \delta_{i,j}$, $(e_{-i},e_j) =
\phi \delta_{i,j}$ for  $i,j \in \{1,\dots,n\}$.  Let $\tilde G =
G^\phi_{2n} = \{x \in \GL_{2n} \mid (xv,xv') = (v,v') \text{ for all
} v,v' \in V\}$, and $\g = \g_{2n}^\phi = \{x \in \gl_{2n} \mid
(xv,v') = -(v,xv') \text{ for all } v,v' \in V\}$ be the Lie algebra
of $\tilde G$. So $\tilde G = \O_{2n}$ and $\g = \so_{2n}$ if $\phi
= +$, and $\tilde G = \Sp_{2n}$ and $\g = \sp_{2n}$ if $\phi = -$.
We write $G$ for the derived group of $\tilde G$, so $G = \tilde G$
in the type C case, and $G = \SO_{2n}$ in the type D case; note that
equivalently $G$ is the identity component of $\tilde G$.

Let $\{e_{i,j} \mid i,j = -n,\dots,-1,1,\dots,n\}$ be the standard
basis of $\gl_{2n}$, and define $f_{i,j} = e_{i,j} -
\eta_{i,j} e_{-j,-i}$ where $\eta_{i,j} = 1$ if $i$ and $j$ have the
same sign and $\eta_{i,j} = \phi$ is $i$ and $j$ have different
signs.  Then the standard basis of $\g$ is $\{ f_{i,j} \mid i + j <
0\}$ if $\phi = +$ and $\{ f_{i,j} \mid i + j \leq 0\}$ if $\phi =
-$. Let $\t = \lan f_{i,i} \mid i = 1,\dots,n \ran$ be the standard
Cartan subalgebra of $\g$ of diagonal matrices. We define
$\{\epsilon'_i \mid i = 1,\dots,n\}$ to be the basis of $\t^*$ dual to
$\{f_{i,i} \mid i = 1,\dots,n\}$ and let $\epsilon_i =
-\epsilon'_i$.

Recall that $W$ denotes the Weyl group of $\g$ with respect to $\t$.
Let $T$ be the maximal torus of $\tilde G$ corresponding to $\t$, and let
$\tilde W = N_{\tilde G} (T)/T$.
Note that in the case $\mf{g} = \mf{so}_{2n}$,  $\tilde W$ is a
Coxeter group of type $\mathrm C_n$ which contains $W$ as a subgroup
of index $2$.

We recall that nilpotent $\tilde G$-orbits in $\g$ are parameterized by
partitions $\bp$, such that each even (respectively odd) part of $\bp$ has even
multiplicity when $\g = \so_{2n}$ (respectively $\sp_{2n}$).  For $\g = \so_{2n}$,
we also recall that a nilpotent $\tilde G$-orbit parameterized by $\bp$ is a
single $G$-orbit unless all parts of $\bp$ are even and of even
multiplicity.  In this latter case, where we say that $\bp$ is {\em
very even}, and the $\tilde G$-orbit parameterized by $\bp$ splits into two
$G$-orbits.

We recall the structure of the component group $\tilde C$ of the
centralizer of $e$ in $\tilde G$.  Suppose $e \in \g$ lies in the
nilpotent $\tilde G$-orbit corresponding to the partition $\bp$.
Then $\tilde C \iso \Z_2^d$, where $d$ is the number of distinct
even parts of $\bp$ if $\g = \sp_{2n}$ and the number of distinct
odd parts of $\bp$ if $\g = \so_{2n}$, see for example
\cite[\S3.13]{Ja}.  We write $C$ for the component group of the
centralizer of $e$ in $G$.  We note that $C$ is equal to $\tilde C$
unless $\g = \so_{2n}$ and $\bp$ has an odd part, in which case $C$
has index 2 in $\tilde C$.

For the remainder of the paper we fix an even multiplicity partition
$\bp = (p_1^2,\dots,p_r^2)$ of $2n$, where $p_i \ge p_{i+1}$ for
each $i$. The {\em symmetric pyramid} of $\bp$ is defined in the
introduction, or equivalently it is the symmetric pyramid with row
lengths given by $\bp$ as defined in \S\ref{ss:stable}. The table
with frame the symmetric pyramid of $\bp$ and with boxes filled by
$-n,\dots,-1,1,\dots,n$ from left to right and top to bottom is
called the {\em coordinate pyramid associated to $\bp$} and denoted
by $\coord(\bp)$; an example of a coordinate pyramid is given in the
introduction.

More generally we say an s-frame $F$ is associated to $\bp$ if
$\part(F) = \bp$. We define the {\em coordinate table with frame
$F$} to be the element $\coord(F) \in \sTab^\leq(F)$ with boxes
filled by $-n, \dots, -1, 1, \dots, n$ such that
\begin{itemize}
\item we obtain $\coord(F)$ from  $\coord(\bp)$ by
rearranging rows and keeping entries in the same boxes, and
\item whenever $F$ has 2 rows of the same length, then the entries in the lower row of $\coord(F)$
are greater than those in the higher row.
\end{itemize}
It is easy to see that these conditions determine $\coord(F)$
uniquely.
For example
\begin{equation*}
    \begin{array}{c}
\begin{picture}(60,80)
\put(0,0){\line(1,0){60}} \put(0,20){\line(1,0){60}}
\put(10,40){\line(1,0){40}} \put(0,60){\line(1,0){60}}
\put(0,80){\line(1,0){60}} \put(0,0){\line(0,1){20}}
\put(20,0){\line(0,1){20}} \put(40,0){\line(0,1){20}}
\put(60,0){\line(0,1){20}}
\put(10,20){\line(0,1){40}} \put(30,20){\line(0,1){40}}
\put(50,20){\line(0,1){40}}
\put(0,60){\line(0,1){20}} \put(20,60){\line(0,1){20}}
\put(40,60){\line(0,1){20}} \put(60,60){\line(0,1){20}}
\put(30,40){\circle*{3}}
\put(8,70){\makebox(0,0){{-3}}}
\put(28,70){\makebox(0,0){{-2}}}
\put(48,70){\makebox(0,0){{-1}}}
\put(18,50){\makebox(0,0){{-5}}}
\put(38,50){\makebox(0,0){{-4}}}
\put(20,30){\makebox(0,0){{4}}}
\put(40,30){\makebox(0,0){{5}}}
\put(10,10){\makebox(0,0){{1}}}
\put(30,10){\makebox(0,0){{2}}}
\put(50,10){\makebox(0,0){{3}}}
\end{picture}
\end{array}
\end{equation*}
is a coordinate table.

Let $F$ be an s-frame associated to $\bp$ and $\coord(F)$ be the
coordinate table with frame $F$. In order to define some elements
and subalgebras of $\g$ associated to $\bp$, we need to fix an
explicit embedding of $\coord(F)$ in the plane. To do this we
declare that the central point is the origin and the boxes have size
$2 \times 2$.  Then given $i \in \{\pm 1, \dots, \pm n\}$, let
$\col(i)$ be the $x$-coordinate of the centre of the box labelled by
$i$, however we use $\row(i)$ to denote the row containing $i$ as
indicated by the labelling of rows from \S\ref{ss:stable}.

We define the nilpotent element $e$ with Jordan type $\bp$ by $e =
\sum f_{i,j}$, where we sum over all $i,j$ such that $i,j$ are
positive and $j$ is in the box immediately to the right of $i$, and
define $h  = \sum_{i=1}^n - \col(i) f_{i,i}$.  For the example above we
have $e = f_{1,2} + f_{2,3} + f_{4,5}$ and $h =
2f_{1,1}-2f_{3,3}+f_{4,4}-f_{5,5}$.
Now the $\ad h$ eigenspace decomposition is given by
\begin{equation*}
\g(k) = \lan
f_{i,j} \mid \col(j) - \col(i) = k \ran,
\end{equation*}
and we can find $f \in \g(-2)$ such that $(e,h,f)$ is an $\sl_2$-triple.
We define the subspaces $\p$, $\h$, $\n$ and $\k$ of $\g$
as in \S\ref{ss:Wdef}, and now we have
$$
\p = \lan f_{i,j} \mid \col(i) \le \col(j) \ran, \qquad \n = \lan
f_{i,j} \mid \col(i) > \col(j) \ran,
$$
$$
\h = \lan f_{i,j} \mid \col(i) = \col(j) \ran , \qquad \k = \lan
f_{i,j} \mid \col(i) = \col(j) +1 \ran.
$$
This gives all the information needed to define the finite $W$-algebra $U(\g,e)$ as in
\S\ref{ss:Wdef}.

We note that $\t$, $e$ and $\h$ are chosen so that $h \in \t$ and
that $\t^e$ is a maximal toral subalgebra  of $\g^e$; $\t^e$ has a basis
given by the elements $\sum_{\row(j) = i} f_{j,j}$ as $i$ ranges over
the rows in the lower half of $F$.  The Levi subalgebra $\g_0$ is
spanned by $\{f_{i,j} \mid \row(i) = \row(j)\}$. We have $\b_0 = \p_0
= \p \cap \g_0$ is a Borel subalgebra of $\g_0$.

The definitions above only depend on $\bp$ and not on the choice of frame
$F$ associated to $\bp$.  Next we define the
parabolic algebra subalgebra
\begin{equation} \label{parabolicq}
\q_F  = \lan f_{i,j} \mid  \row(i) \le \row(j)\ran.
\end{equation}
Then $\g_0$ is a Levi subalgebra of $\q_F$, and we define $\b_F$ to
be the Borel subalgebra of $\g$ generated by $\b_0$ and the
nilradical of $\q_F$.

In the case $F$ is the symmetric pyramid associated to $\bp$, we omit the
subscript $F$, so we just write $\q$ and $\b$. In
particular $\b = \lan f_{i,j} \mid i \le j\ran$ is the Borel
subalgebra consisting of upper triangular matrices in $\g$.

To $A \in \sTab^\le_\phi(F)$ we associate $\lambda_A \in \t^*$ by
declaring that $\lambda_A = \sum_{i =1}^n a_i \eps_i$ where $a_i$ is
the number in the box of $A$ occupying the same position as $-i$
in $\coord(F)$.  For example, if $A$ is the s-table in
\eqref{exstable}, then $\lambda_A  = -7 \eps_6 + 3 \eps_5 -8 \eps_4
-4 \eps_3 + 2 \eps_2 + 5 \eps_1$. The $W_0$-orbit of $\lambda_A$ is
denoted by $\Lambda_A \in \mf{t}^*/W_0$. Thus we can associate to
$A$ the highest weight module $L(\Lambda_A,\q_F)$ as defined in
\S\ref{ss:recaphw}; we denote $L(\Lambda_A, \q_F)$ by $L(A)$ for
short noting that $A$ encodes the parabolic subalgebra $\q_F$. We
define  $\sTab^+_\phi(F)$ to be the subset of $\sTab^\le_\phi(F)$
consisting of s-tables $A$ such that $L(A)$ is finite dimensional.
We write $\Pyr_\phi(\bp)$,
$\Pyr_\phi^+(\bp)$,
and
$\Pyr_\phi^\leq(\bp)$,
for $\sTab_\phi(F)$,
$\sTab^+_\phi(F)$,
and
$\sTab^\leq_\phi(F)$ respectively
when $F$ is the symmetric pyramid associated
to $\bp$.

Let $A \in \sTab^\le_\phi(F)$.  Then the weight $\lambda_A$
satisfies the condition $\lan \lambda_A , \alpha^\vee \ran \notin
\Z_{> 0}$ for all $\alpha \in \Phi_0^+$, because the rows of $A$ are
weakly increasing. The condition that the entries of $A$ either all
lie in $\Z$ or all lie in $\Z + \frac{1}{2}$ (the latter only if
$\g$ is of type D) implies that $\lambda_A \in \t^*_\Z$.

\subsection
{Associated varieties of primitive ideals} \label{ss:BV}

In this subsection we recall from \cite{BV} how to calculate the
associated variety of a primitive ideal in the universal enveloping
algebras of the classical Lie algebras.  For our purposes we
restrict to $\g$ of type C or D. For
a primitive ideal $I$ of $U(\g)$ we recall that the associated
variety $\cVA(I)$ of $I$ is the closure of a nilpotent $G$-orbit,
see for example \cite[\S9]{Ja}.

By Duflo's Theorem (see \cite{Du}), it suffices to calculate the
associated variety of annihilators of irreducible highest weight
modules. These modules are defined in terms of the Borel subalgebra
$\b$ and the Cartan subalgebra $\t \sub \b$ from \S\ref{ss:hwnotn}.
In this paper, we only consider integral weights $\lambda \in
\t_\Z^*$.  For such $\lambda$, we write $L(\lambda)$ for the
irreducible highest weight $U(\g)$-module with highest weight
$\lambda-\rho$ with respect to $\b$.  We recall that we say that
$\lambda$ is {\em anti-dominant} if $\langle \lambda,\alpha^\vee
\rangle \in \Z_{\le 0}$ for all $\alpha \in \Phi^+$. Also we recall
that for any $\mu \in \t^*_\Z$ there exists $w \in W$ and
antidominant $\lambda \in \t^*_\Z$ such that $\mu = w \lambda$.

A weight $\lambda \in \t^*$ is {\em regular} if
$\langle \lambda , \alpha^\vee \rangle \neq 0$
for all $\alpha \in \Phi$.
For any regular, anti-dominant weights $\lambda_1$ and $\lambda_2$,
and $w \in W$ it is well known that
\begin{equation} \label{e:assvar}
\cVA(\Ann L(w \lambda_1)) = \cVA(\Ann L(w \lambda_2)),
\end{equation}
see for example \cite[\S9.12]{Ja}. Recall that $\tilde W = N_{\tilde
G} (T)/T$, so when $\g$ is of type D it contains $W$ as a subgroup
of index 2. Suppose $\g$ is of type D and let $s \in \tilde W$ be
the element that fixes all $\eps_i$ except for $\eps_1$, which it
sends to $-\eps_1$, so $s \notin W$. If $\lambda \in \t^*_\Z$ is
antidominant, then it is easy to check that $s\lambda$ is too. It
follows that $\cVA(\Ann L(w \lambda)) = \cVA(\Ann L(ws \lambda))$
for any $w \in W$, so \eqref{e:assvar} holds for all $w \in \tilde
W$.

The following is an algorithm adapted from \cite{BV} to determine
the associated variety of a primitive ideal in types C and D.
\begin{Algorithm} \label{A:BV2}
\noindent {\em Input:} $\lambda = \sum a_i \eps_i \in \t^*_\Z$.

\noindent {\em Step 1:} Let $\bq = \RS(a_n, \dots, a_1, -a_1, \dots,
-a_n)$.
In type D when calculating the Robinson-Schensted algorithm,
if zeros occur, then we treat the two zeros closest
to the middle of $(a_n, \dots, a_1, -a_1, \dots, -a_n)$
as if the first zero is larger than the second.

\noindent {\em Step 2:} Put $\bq = (q_1 \le q_2 \le \dots \le q_m)$
into ascending order. By inserting zero into $\bf{q}$ if necessary,
in type C assume that $\bq$ has an odd number of parts, and in type
D assume that $\bq$ has an even number of parts.  For $i =1 \dots m$
let $r_i = q_i+i-1$ to create the list $(r_i)$. Let $(2 s_1, \dots,
2 s_{l})$ be the sublist of $(r_i)$ consisting of even numbers, and
let $(2 t_1+1, \dots, 2 t_{k} +1)$ be the sublist of $(r_i)$
consisting of odd numbers.

\noindent {\em Step 3:} Let $(u_i)$ be list obtained by sorting the
concatenation of $(s_i)$ with $(t_i)$. Now let $(s'_i)$ and $(t'_i)$
be the sublists of $(u_i)$ consisting of the terms with odd and even
indices, respectively. Let $(r'_i)$ denote the list obtained by
sorting the concatenation of $(2 s'_i)$ and $(2t'_i+1)$ in type C,
and $(2 s'_i +1)$ and $(2t'_i)$ in type D. Finally let $q'_i =
r'_i +1 -i$ to form the partition $\bf{q}'$.
\end{Algorithm}

The following corollary is a consequence of \cite[Theorem 18]{BV}.

\begin{Corollary}  \label{C:BV2}
Let $\lambda \in \t^*_\Z$ and let $\bq'$ be the output of
Algorithm~\ref{A:BV2}.
\begin{enumerate}
\item[(i)]  Suppose $\g$ is of type C or $\g$ is of type D and $\bp'$ is
not very even.  Then $\cVA(\Ann L(\lambda))$ is equal to the
closure of the
nilpotent
$\tilde G$-orbit corresponding to the partition $\bq'$.
\item[(ii)]  Suppose $\g$ is of type D and $\bp'$ is very even.
Then $\cVA(\Ann L(\lambda))$ is equal to the closure of
one of the two $G$-orbits in the
nilpotent $\tilde G$-orbit corresponding to the partition $\bq'$.
\end{enumerate}
\end{Corollary}

\begin{proof}
    To prove Corollary \ref{C:BV2} we present the algorithm given in \cite{BV}.  This is broken
in to four steps.  After each step we make some remarks and explain
and justify some adaptations that we make to get Algorithm
\ref{A:BV2}.

\medskip

\noindent {\em Input:} The algorithm from \cite{BV} takes as input
an element $w \in \bar W$ and calculates the associated variety of
the annihilator of $L(w\lambda)$ where $\lambda \in \t^*_\Z$ is
antidominant and regular.  Let $\sigma = n \eps_n + (n-1)\eps_{n-1} + \dots +
\eps_1 \in \t^*$, and identify $w$ with $w \sigma$. Then identify $w
\sigma = \sum_{i=1}^n a_i \eps_i$ with the list $(a_n, \dots, a_1,
-a_1, \dots, -a_n)$.

\noindent {\em Adaptation:}  Instead in Algorithm \ref{A:BV2} we take as input $\mu = w
\lambda = \sum_{i=1}^n b_i \eps_i$ and we identify $\mu$ with the
list $(b_n, \dots, b_1, -b_1, \dots, -b_n)$.  This is justified
below.

\noindent {\em Step 1:} Calculate $\part(\RS(a_n, \dots, a_1,-a_1,
\dots, a_n))$, and let $\bf{q}$ be the transpose partition.

\noindent {\em Remarks and adaptation:}  In \cite{BV}
Barbasch and Vogan use a ``dual'' version of the Robinson-Schensted
algorithm which results in the transpose partition to that obtained
from the version given in \S\ref{ss:RS}, so they do not need to take
the transpose here.

It follows from Lemma~\ref{altcRS} that $\part(\RS(a_n, \dots, a_1,
-a_1, \dots, -a_n))$ is the transpose of $\part(\RS(-a_n, \dots,
-a_1, a_1, \dots, a_n))$.  Therefore, we obtain $\bq$ by applying
the Robinson--Schensted algorithm to $(-a_n, \dots, -a_1, a_1,
\dots, a_n)$.

Note that $-\sigma$ is antidominant, and $w(-\sigma) = -\sigma w =
\sum_{i=1}^n -a_i \eps_i$.  This means that for $i,j  = 1, \dots, n$
we have $-a_i < -a_j$ precisely when $b_i < b_j$ and
$-a_i < a_j$ precisely when $b_i < -b_j$. This implies that
$$\part(\RS(-a_n, \dots, -a_1, a_1,\dots, a_n)) =
\part(\RS(b_n, \dots, b_1, -b_1, \dots, -b_n)).$$
So in Algorithm \ref{A:BV2} we instead calculate $\bq = \part(\RS(b_n, \dots, b_1, -b_1,
\dots, -b_n))$.

\noindent {\em Step 2:} Put $\bq$ into ascending order; say $\bq =
(q_1 \le q_2 \le \dots \le  q_{2k +1})$ by inserting a zero
if necessary to ensure there are an odd number of parts.  For
$i =1 \dots 2k+1$, let $r_i = q_i+i-1$ to create the list $(r_i)$.
Let $(2 s_1, \dots, 2 s_{k+1})$ be the sublist of $(r_i)$ consisting
of even numbers, and let $(2 t_1+1, \dots, 2 t_k +1)$ be the sublist
of $(r_i)$ consisting of odd numbers.

\noindent {\em Step 2a:} Only do the following in the type D case.
If $s_1 \neq 0$, then replace the list $(t_1, \dots, t_k)$ with the
list $(0, t_1+1, \dots, t_k+1)$. If $s_1 = 0$, then replace the list
$(s_1, \dots, s_{k+1})$ with the list $(s_2-1, \dots, s_{k+1}-1)$.

\noindent {\em Remarks and adaptation:} The fact that the list
$(s_i)$ has $k+1$ elements and the list $(t_i)$ has $k$ elements
(before the change in the type D case) is due to
\cite[Proposition 17]{BV}.

In the type D case, suppose instead that we modify $\bq$ by possibly
adding a zero to the start to ensure that it has an even number
entries.  Then we can define the lists $(s_i)$ and $(t_i)$ using
the same procedure.  One can check that these lists are
exactly the same as those obtained by assuming that $\bq$ has an odd
number of entries and then doing Step 2a.  Therefore, we remove this
step in Algorithm~\ref{A:BV2}.

\noindent {\em Step 3:}
Do this step exactly as Step 3 in Algorithm \ref{A:BV2}.

\noindent {\em Output:} The nilpotent $\tilde G$-orbit corresponding to the
partition $\bq'$.

By \cite[Theorem 18]{BV} the
partition $\bq'$ corresponds to a nilpotent $\tilde G$-orbit,
and when $\g$
is of type C, then $\cVA(\Ann L(w\lambda))$ is precisely this
orbit. In the case $\g$ is of type D and $\bq'$ is not very even,
then \cite[Theorem 18]{BV} gives that
$\cVA(\Ann L(w\lambda))$ is the closure of the nilpotent $\SO_{2n}$-orbit
corresponding to $\bq'$. If $\bq'$ is very even,
then \cite[Theorem 18]{BV} gives that
the associated
variety is the closure of
one of the $\SO_{2n}$-orbits contained in the $\tilde G$-orbit
corresponding to $\bq'$.
We do not need to know which orbit for
our purposes.

Our convention about zeros in the type D case from Algorithm \ref{A:BV2}
is due to \cite[p. 173]{BV}.

So far we have only justified that Algorithm \ref{A:BV2} works
the case that $\lambda$ is regular, however we can remove this assumption
using \cite[Lemma 2.4]{Jo2}.
\end{proof}

We recall that the array $\binom{s}{t}$ obtained by placing the list
$(s_i)$ on the top row and the list $(t_i)$ is called a {\em symbol}
associated to $\bq$.
This notation was introduced by Lusztig in
\cite{Lu}.
We call the concatenated list $(u_i)$
the {\em content} of the symbol $\binom{s}{t}$, or simply the
{\em content} of $\bq$.

\subsection{The component group action} \label{ss:comp}
At present the only cases where
the nontrivial action of
$\tilde C$ on the set of isomorphism classes of finite dimensional
irreducible $U(\g,e)$-modules is known is where $\mf{g} =
\mf{sp}_{2n}$ or $\mf{g} = \mf{so}_{2n}$ and the Jordan type of $e$
has an even number of Jordan blocks all of the same size, i.e.\ the
case $\bp = (l^{2r})$, so that the symmetric pyramid of $\bp$ is a
rectangle. The description of the action depends on the notion of the
{\em $\sharp$-element} of a list of complex numbers.

Given a list $(a_1,\dots,a_{2k+1})$ of complex numbers
let $\{(a_1^{(i)},\dots,a_{2k+1}^{(i)}) \mid i \in I\}$ be the set
of all permutations of this list which satisfy
$a^{(i)}_{2j-1}+a^{(i)}_{2j} \in \Z_{>0}$
for each $j=1,\dots,k$. Assuming that such rearrangements exist, we
define the {\em $\sharp$-element} of $(a_1,\dots,a_{2k+1})$ to be
the unique maximal element of the set $\{a_{2k+1}^{(i)} \mid i \in
I\}$. On the other hand, if no such rearrangements exist, we say
that the $\sharp$-element of $(a_1,\dots,a_{2k+1})$ is undefined.
For example, the $\sharp$-element of $(-3, -1, 2)$ is $-3$, whereas
the $\sharp$-element of $(-3,-2,1)$ is undefined.

We abuse notation somewhat about saying that the $\sharp$-element of
a list of numbers with an even number of elements is the
$\sharp$-element of that list with zero inserted.

The following lemma is easy to prove and is required in the proof of
Theorem \ref{T:main}.

\begin{Lemma} \label{L:sharp1}
    If $(a_1, \dots, a_{2k})$ is a list of integers which satisfies
    $a_{2i-i} + a_{2i} > 0$ for $i=1,\dots,k$, and $b_1, \dots, b_{2k}$
    is same list sorted into weakly increasing order, then
    $b_{i}+b_{2k+1-i} > 0$ for $i=1,\dots, 2k$.
\end{Lemma}

For the next few paragraphs let $F$ denote the unique s-frame which
satisfies $\part(F) = (l^{2r})$, where $r$ and $l$ are fixed
positive integers, and $l$ is even if $\g = \sp_{2n}$ and $l$ is odd
if $\g = \so_{2n}$. In this case we have $\tilde C \iso \Z_2 = \lan c \ran$
and we define an operation of $c$ on $\sTab^+(F) \sub \sTab^\leq(F)$
as follows. Let $A \in \sTab^+(F)$ and let $a_1, \dots, a_l$ be row
$-1$ of $A$. By \cite[Theorem 1.2]{Bro2} the $\sharp$-element
of $a_1, \dots, a_l$ is defined; let $a$ be this number. We declare
that $c \cdot A \in \sTab^\leq(F)$ is the s-table with the same rows
as $A$, except with one occurrence of $a$ replaced with $-a$ in row
$-1$, and one occurrence of $-a$ replaced with $a$ in row $1$. Then
\cite[Theorem 1.3]{Bro2} says that $c \cdot L(A) = L(c \cdot A)$. An
example of this action  is
\begin{equation*}
    c \cdot
    \begin{array}{c}
\begin{picture}(40,40)
\put(0,0){\line(1,0){40}}
\put(0,20){\line(1,0){40}} \put(0,40){\line(1,0){40}}
\put(0,0){\line(0,1){20}} \put(20,0){\line(0,1){20}}
\put(40,0){\line(0,1){20}} \put(0,20){\line(0,1){20}}
\put(20,20){\line(0,1){20}} \put(40,20){\line(0,1){20}}
\put(8,10){\makebox(0,0){{-2}}} \put(28,10){\makebox(0,0){{-1}}}
\put(10,30){\makebox(0,0){{1}}} \put(30,30){\makebox(0,0){{2}}}
\put(20,20){\circle*{3}}
\end{picture}
\end{array}
=
    \begin{array}{c}
\begin{picture}(40,40)
\put(0,0){\line(1,0){40}}
\put(0,20){\line(1,0){40}} \put(0,40){\line(1,0){40}}
\put(0,0){\line(0,1){20}} \put(20,0){\line(0,1){20}}
\put(40,0){\line(0,1){20}} \put(0,20){\line(0,1){20}}
\put(20,20){\line(0,1){20}} \put(40,20){\line(0,1){20}}
\put(8,10){\makebox(0,0){{-1}}} \put(30,10){\makebox(0,0){{2}}}
\put(8,30){\makebox(0,0){{-2}}} \put(30,30){\makebox(0,0){{1}}}
\put(20,20){\circle*{3}}
\end{picture}
\end{array}
.
\end{equation*}

The next lemma helps explain what happens when the
Robinson-Schensted algorithm is applied to $A$ in the case that $r =
2$ and the $\sharp$-element of row $-1$ of $A$ is defined; it is
required in the proof of Theorem~\ref{T:main}.
For this lemma, in the case that $l$ is odd, we use the zero convention
for calculating the Robinson-Schensted algorithm from
Algorithm \ref{A:BV2}.
\begin{Lemma} \label{L:sharpdef}
    Let $A$ be a rectangular s-frame with
   distinct entries and with
    $\part(A) = (l^2)$.
     Then the
    $\sharp$-element of row $-1$ of $A$ is defined
    if and only if $\RS(A) = (l,l)$ or $\RS(A) = (l+1,l-1)$.
    More specifically,
    let $a$ be the $\sharp$-element of row $-1$ of $A$.
    If $a\ge 0$, then $\RS(A) = (l,l)$.
    If $a <  0$, then $\RS(A) = (l+1,l-1)$.
 \end{Lemma}
 \begin{proof}
     Let $a_1, \dots, a_{l}$ be row $-1$ of $A$, and suppose that the $\sharp$-element of
     $(a_1, \dots, a_{l})$ is defined and is $a_s$.
     Suppose that $a_s\ge0$.
     If $a_s >0$ or if $l$ is even,
     then by Lemma~\ref{L:sharp1} we have that
     $a_i + a_{l+1-i} > 0$ for $i=1, \dots, l$.
     Thus for $i=1, \dots, l$ we have
     that $a_i > -a_{l+1-i}$, which implies that
     $\RS(A) = (l,l)$.
     If $a_s = 0$ and $l$ is odd, then we must have that
     $s = (l+1)/2$, and we have that
     $a_i + a_{l+1-i} > 0$ for $i=1, \dots, (l-1)/2$.
     Thus for $i=1, \dots, (l-1)/2, (l+3)/2, \dots, l$,  we have
     that $a_i > -a_{l+1-i}$.
     Since we also count $0 = a_s$ as greater than $0 = -a_s$ when
     calculating the Robinson-Schensted algorithm,
     we have that
     $\RS(A) = (l,l)$.
     Now suppose that $a_s < 0$.
      Then
     $1 \le s \le (l+1)/2$.  Since $a_s$ is the sharp element,
$a_1,\dots,a_s,-a_{l+1-s},\dots,-a_1$ is an increasing sequence of
length $l+1$ in $\word(A)$.  Combining this with the fact that the rows
of $A$ are increasing and using Lemma~\ref{altRS}, we see that
$\RS(A) \ge (l+1,l-1)$. Also by Lemma \ref{L:sharp1} we
     must have that
     \begin{itemize}
\item $a_j+a_{l+1-j} > 0$ for $1 \le j < s$,
\item $a_j+a_{l+2-j} > 0$ for $s < j \le (l+1)/2$, and
\item $a_{l/2+1} > 0$ if $l$ is even.
     \end{itemize}
     The only way that $\RS(A)$ could be larger
        than $(l+1,l-1)$ in the dominance order is if the first term is larger than
        $l+1$.  By Lemma~\ref{altRS}, this is only possible if $\word(A)$
    contains a weakly increasing subsequence
        of length at least $l+2$, which can only happen if there exists $j$
        such that $a_j < -a_{l+2-j}$, which cannot happen due to the
        above conditions.

        Now suppose that $\part(\RS(A)) = (l,l)$.
        Then we must have that $a_i + a_{l+1-i} > 0$ for all $i$
    such that $1 \leq i < (l+1)/2$,
        so the $\sharp$-element is defined.

        Finally suppose that $\part(\RS(A)) = (l+1,l-1)$.
        So for some $j$ we have that

\begin{equation} \label{rsa}
  \RS(A) =
  \begin{array}{|c|c|c|c|c|c|c|c|c|c|}
      \cline{1-8}
     a_1 & a_2 & \dots & a_{j-1} & a_{j+1} & a_{j+2} & \dots & a_{l}  \\
     \hline
     -a_{l} & -a_{l-1} & \dots & -a_{l-j+2} & a_j &
      -a_{l-j+1} & \dots & -a_3 & -a_2 & -a_1 \\
     \hline
  \end{array}
\end{equation}
This implies that $j \le l/2$, otherwise we would have $-a_j < a_{l
-j+1}$ and $-a_j < a_j < -a_{l-j+1}$, which is a contradiction. Also
\eqref{rsa} tells us that the following sums are all positive
integers: $a_1 + a_{l}, a_2 + a_{l-1}, \dots, a_{j-1} + a_{l-j+2}$.
We also have that $a_{j+1} + a_{l-j+1}, \dots, a_l + a_2$  are all
positive, because during the Robinson-Schensted algorithm
$-a_{l-j+1}$ must bump $a_{j+1}$, $-a_{l-j}$ must bump $a_{j+2}$,
and so on. Thus the $\sharp$-element of row $-1$ of $A$ is defined.
\end{proof}

\begin{Remark} \label{R:sharp}
    It is useful in the proof of Theorem \ref{T:main} to consider explicitly calculating $\RS(c \cdot A)$
when $A$ has 2 rows, the $\sharp$-element of $A$ is positive, and
all the elements of row $-1$ of $A$ are distinct. Let $a_1, \dots,
a_l$ be row $-1$ of $A$, and let $a_s$ be the $\sharp$-element of
$a_1, \dots, a_l$. So we must have that $s > l/2$. If $s > (l+1)/2$,
then by calculating the Robinson-Schensted algorithm on both words,
observe that the word $a_1, \dots, a_l, -a_l, \dots, -a_1$ is
Knuth-equivalent to
\begin{align*}
    a_1, \dots, &a_{l-s}, -a_l, \dots, -a_{s+1},a_{l-s+1},-a_s, a_{l-s+2}, \dots, a_{s-1}, \\
    &-a_{s-1} ,\dots, -a_{l-s+2}, a_s, -a_{l-s+1}, a_{s+1}, \dots, a_l, -a_{l-s}, \dots, -a_1.
\end{align*}
By swapping $a_{l-s+1}$ with $-a_s$, and $a_s$ with $-a_{l-s+1}$
we get the word
\begin{align*}
a_1, \dots, &a_{l-s}, -a_l, \dots, -a_{s},a_{l-s+1}, \dots, a_{s-1}, \\
  &-a_{s-1} ,\dots, -a_{l-s+1}, a_{s}, \dots, a_l, -a_{l-s}, \dots, -a_1.
\end{align*}
This in turn
is Knuth-equivalent to the word
\begin{align*}
a_1, \dots, &a_{l-s}, -a_s, a_{l-s+1}, \dots, a_{s-1}, a_{s+1}, \dots, a_l, \\
  &-a_l, \dots, -a_{s+1}, -a_{s-1}, \dots, -a_{l-s+1}, a_s,-a_{l-s}, \dots, -a_1,
\end{align*}
which is $\word(c \cdot A)$.

If $s = (l+1)/2$, then
by calculating Robinson-Schensted on both words, observe that the word $a_1, \dots, a_l, -a_l, \dots, -a_1$
is Knuth-equivalent to
\[
    a_1, \dots, a_{s-1}, -a_l, \dots, -a_{s+1},a_s,
    -a_s, a_{s+1}, \dots, a_l, -a_{s-1}, \dots, -a_1.
    \]
By swapping $a_{s}$ with $-a_s$
we get the word
\[
a_1, \dots, a_{s-1}, -a_l, \dots, -a_{s},
  a_{s}, \dots, a_l, -a_{s-1}, \dots, -a_1.
  \]
This in turn
is Knuth-equivalent to the word
\[
a_1, \dots, a_{s-1}, -a_s, a_{s+1}, \dots, a_l,
-a_l, \dots, -a_{s+1}, a_{s}, -a_{s-1}, \dots, -a_1,
  \]
which is $\word(c \cdot A)$.
\end{Remark}

Now we are in a position to
describe an operation of the elements of $\tilde C$ on $\Pyr_\phi^+(\bp)$
for an arbitrary even multiplicity partition $\bp$.
\begin{Remark}
In this work we
do not verify that
the operation of elements of $\tilde C$ on
$\Pyr_\phi^+(\bp)$
defines a group action of $\tilde C$, however this is the case.
While we have a proof of this, it is rather lengthy.
Moreover, in the future work \cite{BroG} we will show that
the operation of elements of $\tilde C$ on  $\Pyr_\phi^+(\bp)$ corresponds to
the action of $\tilde C$ on finite dimensional irreducible $U(\mf{g},e)$-modules.
This will imply that we do have a well defined $\tilde C$-action on
 $\Pyr_\phi^+(\bp)$.
 With the exception of Corollary \ref{C:evencase}, all of our results holds without knowing that the
operation of elements of $\tilde C$ on $\Pyr_\phi^+(\bp)$ defines a group action.
 \end{Remark}

 Let $i_1, \dots, i_d$ be such that $i_j < i_{j+1}$ and
$p_{i_1},\dots,p_{i_d}$ are  the minimal the distinct parts of
$\bp = (p_1^2 \ge p_2^2 \ge \dots \ge p_r^2)$
that are
odd (respectively even) when $\g = \so_{2n}$ ( respectively $\sp_{2n}$).
By minimal we mean that if $p_k = p_{i_j}$, then $k \ge i_j$.
Then we can choose generators
$c_1,\dots,c_d$ for $\tilde C \iso \Z_2^d$ corresponding to
$p_{i_1},\dots,p_{i_d}$.
More specifically,
in type C
we set
\[
c_k = \sum_{\substack{
         -n \le i,j \le n \\ \col(i) = \col(j) \\ \row(i) = i_k \\ \row(j) = -i_k}}
    \sgn(\col(i)) (e_{i,j} + e_{j,i})
    + \sum_{\substack{ -n \le i \le n \\ \row(i) \neq \pm i_k}} e_{i,i},
\]
and in type D we set
\[
c_k = \sum_{\substack{
         -n \le i,j \le n \\ \col(i) = \col(j) \\ \row(i) = i_k \\ \row(j) = -i_k}}
     (-1)^{\col(i)/2} (e_{i,j} + e_{j,i})
    + \sum_{\substack{ -n \le i \le n \\ \row(i) \neq \pm i_k}} e_{i,i}.
\]
Now one can calculate that $c_k \in \tilde H^e$.
Furthermore the argument used in \cite[Section 6]{Bro2}
can be adapted to show that
$\tilde C$ is generated by $c_1, \dots, c_d$.
Note that in the type D case, then any word $w$
in $c_1, \dots, c_d$ of even length lies in $C$.

Next we explain how to extend the operation of $c$ on rectangular
s-tables given above  to any s-table $A \in \sTab_\phi^+(F)$ as it
only involves the middle two rows.  We assume the middle two rows of
$A$ have odd length if $\g = \so_{2n}$ and even length if $\g =
\sp_{2n}$. We consider the Levi subalgebra $\mf{g}^\mf{r}$, where
\[
\mf{r} = \left\lan \sum_{\substack {1 \leq i \leq  n \\
\row(i) \neq 1}}  f_{i,i} \right\ran.
\]
Then $\mf{g}^\r \iso
\mf{gl}_{n-2 l} \oplus \mf{g}^\phi_{2 l}$,
where $l$ is the length of row $1$ of $F$.
Thus by
Corollary~\ref{C:inductsl}, $L^\r(\Lambda_A,\q_F)$ is finite
dimensional. In turn this implies that the irreducible highest
weight module $L(A^{-1}_1)$ for $U(\mf{g}^\phi_{2 l},e')$,
where $e' \in \mf{g}^\phi_{2 l}$ is a nilpotent element with
Jordan type $(l^2)$, is finite
dimensional.  This means that $c \cdot A^{-1}_1$ is defined, so we
can define $c \cdot A$ to be the s-table obtained from $A$ by
replacing the middle two rows with $c \cdot A^{-1}_1$.  By
Lemma~\ref{L:cact1} and \cite[Theorem 6.1]{Bro2} we have that
\begin{equation} \label{EQ:clact}
   L(c \cdot A) \iso c \cdot L(A).
\end{equation}
So in particular, $L(c \cdot A)$ is finite dimensional, so $c \cdot
A \in \sTab_\phi^+(F)$.

To define the operation $c_k$  on $\Pyr_\phi^+(F)$ we require the operators $\bar s_i$
defined in \S\ref{ss:stable}. Another important Levi subalgebra is
$\g^\s$, where
\[
   \mf{s} = \left\lan \sum_{i=1}^n f_{i,i} \right\ran.
\]
Then $\mf{g}^\mf{s} \iso \mf{gl}_n$.  The next lemma is required to
ensure the operation of $c_k$ is defined.

\begin{Lemma} \label{L:siwelldef}
    Let $F$ be an s-frame associated to $\bp$ and
    let $A \in \sTab^+_\phi(F)$.
    Then
    \begin{enumerate}
    \item[(i)] $\part (\RS(A^+)) = \part(A^+)$ and
    \item[(ii)] $w \cdot A$ is defined for all words $w$ in $\bar s_1, \dots, \bar s_{r-1}$.
    \end{enumerate}
\end{Lemma}

\begin{proof}
    By Corollary~\ref{C:inductsl} we have that $L^\mf{s}(\Lambda_A,\q_F)$
    is finite dimensional.
    This, along with Corollary~\ref{C:BGKconj} and \cite[Corollary 3.3]{Jo1},
    implies that $\part(\RS(A^+)) = \part(A^+)$ giving (i).  Now (ii)
    follows from Lemmas~\ref{swaprs} and \ref{swapdef}.
\end{proof}

Now we define the action of $c_k$ on $A \in \Pyr_\phi^+(\bp)$. First
we apply $\bar s_{i_k-1},\dots,\bar s_1$ to $A$ moving row $i_k$ in
to the middle.  Each of these operations is defined by
Lemma~\ref{L:siwelldef}.  Next we apply $c$ to obtain $A'$.  Using
\eqref{EQ:clact}, we have that $A' \in \sTab_\phi^+(F)$, where $F$
is the frame obtained from the symmetric pyramid by applying $\bar
s_{i_k-1},\dots,\bar s_1$. We finish by applying the operators $\bar
s_1,\dots,\bar s_{i_k-1}$ so that we end up with an element of
$\Pyr_\phi^\leq(\bp)$. Again by Lemma~\ref{L:siwelldef} each of these
operators is defined. Putting this together we define
\[
  c_k \cdot A = \bar s_{i_k} \bar s_{i_k-1} \dots \bar s_{1} c \bar s_{1} \dots \bar s_{i_k} \cdot A.
\]

\begin{Lemma} \label{L:fdpreserved}
    Let $A \in \Pyr_\phi^+(\bp)$ (so $L(A)$ is finite dimensional), and let
    $w$ be a word in $c_1, \dots, c_d$.
    Then $L(w \cdot A)$ is finite dimensional.
\end{Lemma}
\begin{proof}
    This follows immediately from Lemma~\ref{swaprs}, Corollary~\ref{C:BV2} and \eqref{EQ:clact}.
\end{proof}

Although we do not check here if the operations of the $c_k$ on
$\Pyr_\phi^+(\bp)$ lead to an action of $\tilde C$ on $\Pyr_\phi^+(\bp)$,
we allow ourselves to say that $A,B \in \Pyr_\phi^+(\bp)$ are {\em
$\tilde C$-conjugate} if there is a word $w$ in $c_1, \dots, c_d$ such that
$B = w \cdot A$.

To see some examples of applications of the operator $c_k$, if
\[
A =
\begin{array}{c}
\begin{picture}(80,80)
\put(20,0){\line(1,0){40}} \put(0,20){\line(1,0){80}}
\put(0,40){\line(1,0){80}} \put(0,60){\line(1,0){80}}
\put(20,80){\line(1,0){40}}
\put(20,0){\line(0,1){20}}
\put(40,0){\line(0,1){20}}
\put(60,0){\line(0,1){20}}
\put(60,20){\line(0,1){20}}
\put(80,20){\line(0,1){20}}
\put(60,40){\line(0,1){20}}
\put(80,40){\line(0,1){20}}
\put(20,80){\line(0,-1){20}}
\put(40,80){\line(0,-1){20}}
\put(60,80){\line(0,-1){20}}
\put(0,20){\line(0,1){40}}
\put(20,20){\line(0,1){40}}
\put(40,20){\line(0,1){40}}
\put(40,40){\circle*{3}}
\put(30,70){\makebox(0,0){{2}}}
\put(50,70){\makebox(0,0){{5}}}
\put(10,50){\makebox(0,0){{1}}}
\put(30,50){\makebox(0,0){{3}}}
\put(50,50){\makebox(0,0){{4}}}
\put(70,50){\makebox(0,0){{6}}}
\put(8,30){\makebox(0,0){{-6}}}
\put(28,30){\makebox(0,0){{-4}}}
\put(48,30){\makebox(0,0){{-3}}}
\put(68,30){\makebox(0,0){{-1}}}
\put(28,10){\makebox(0,0){{-5}}}
\put(48,10){\makebox(0,0){{-2}}}
\end{picture}
\end{array},
\]
then
\[
c_1 \cdot A =
\begin{array}{c}
\begin{picture}(80,80)
\put(20,00){\line(1,0){40}} \put(0,20){\line(1,0){80}}
\put(0,40){\line(1,0){80}} \put(0,60){\line(1,0){80}}
\put(20,80){\line(1,0){40}}
\put(20,0){\line(0,1){20}}
\put(40,0){\line(0,1){20}}
\put(60,0){\line(0,1){20}}
\put(60,20){\line(0,1){20}}
\put(80,20){\line(0,1){20}}
\put(60,40){\line(0,1){20}}
\put(80,40){\line(0,1){20}}
\put(20,80){\line(0,-1){20}}
\put(40,80){\line(0,-1){20}}
\put(60,80){\line(0,-1){20}}
\put(0,20){\line(0,1){40}}
\put(20,20){\line(0,1){40}}
\put(40,20){\line(0,1){40}}
\put(40,40){\circle*{3}}
\put(30,70){\makebox(0,0){{2}}}
\put(50,70){\makebox(0,0){{5}}}
\put(8,50){\makebox(0,0){{-6}}}
\put(30,50){\makebox(0,0){{1}}}
\put(50,50){\makebox(0,0){{3}}}
\put(70,50){\makebox(0,0){{4}}}
\put(8,30){\makebox(0,0){{-4}}}
\put(28,30){\makebox(0,0){{-3}}}
\put(48,30){\makebox(0,0){{-1}}}
\put(70,30){\makebox(0,0){{6}}}
\put(28,10){\makebox(0,0){{-5}}}
\put(48,10){\makebox(0,0){{-2}}}
\end{picture}
\end{array}
\]
and by conferring with \eqref{ex1} we see that
\[
c_2 \cdot A = \bar s_1 \cdot
\begin{array}{c}
\begin{picture}(80,80)
\put(0,0){\line(1,0){80}}
\put(0,20){\line(1,0){80}}
\put(20,40){\line(1,0){40}}
\put(0,60){\line(1,0){80}}
\put(0,80){\line(1,0){80}}
\put(0,0){\line(0,1){20}}
\put(20,0){\line(0,1){20}}
\put(40,0){\line(0,1){20}}
\put(60,0){\line(0,1){20}}
\put(80,0){\line(0,1){20}}
\put(60,20){\line(0,1){20}}
\put(60,40){\line(0,1){20}}
\put(20,80){\line(0,-1){20}}
\put(40,80){\line(0,-1){20}}
\put(60,80){\line(0,-1){20}}
\put(20,20){\line(0,1){40}}
\put(40,20){\line(0,1){40}}
\put(0,60){\line(0,1){20}}
\put(80,60){\line(0,1){20}}
\put(40,40){\circle*{3}}
\put(10,70){\makebox(0,0){{2}}}
\put(30,70){\makebox(0,0){{3}}}
\put(50,70){\makebox(0,0){{5}}}
\put(70,70){\makebox(0,0){{6}}}
\put(28,50){\makebox(0,0){{-4}}}
\put(50,50){\makebox(0,0){{1}}}
\put(28,30){\makebox(0,0){{-1}}}
\put(50,30){\makebox(0,0){{4}}}
\put(8,10){\makebox(0,0){{-6}}}
\put(28,10){\makebox(0,0){{-5}}}
\put(48,10){\makebox(0,0){{-3}}}
\put(68,10){\makebox(0,0){{-2}}}
\end{picture}
\end{array}
=
\begin{array}{c}
\begin{picture}(80,80)
\put(20,0){\line(1,0){40}} \put(0,20){\line(1,0){80}}
\put(0,40){\line(1,0){80}} \put(0,60){\line(1,0){80}}
\put(20,80){\line(1,0){40}}
\put(20,0){\line(0,1){20}}
\put(40,0){\line(0,1){20}}
\put(60,0){\line(0,1){20}}
\put(60,20){\line(0,1){20}}
\put(80,20){\line(0,1){20}}
\put(60,40){\line(0,1){20}}
\put(80,40){\line(0,1){20}}
\put(20,80){\line(0,-1){20}}
\put(40,80){\line(0,-1){20}}
\put(60,80){\line(0,-1){20}}
\put(0,20){\line(0,1){40}}
\put(20,20){\line(0,1){40}}
\put(40,20){\line(0,1){40}}
\put(40,40){\circle*{3}}
\put(30,70){\makebox(0,0){{2}}}
\put(50,70){\makebox(0,0){{3}}}
\put(8,50){\makebox(0,0){{-4}}}
\put(30,50){\makebox(0,0){{1}}}
\put(50,50){\makebox(0,0){{5}}}
\put(70,50){\makebox(0,0){{6}}}
\put(8,30){\makebox(0,0){{-6}}}
\put(28,30){\makebox(0,0){{-5}}}
\put(48,30){\makebox(0,0){{-1}}}
\put(70,30){\makebox(0,0){{4}}}
\put(28,10){\makebox(0,0){{-3}}}
\put(48,10){\makebox(0,0){{-2}}}
\end{picture}
\end{array}
.
\]

\subsection{Proof of the classification} \label{ss:proof}

We define
\[
\Pyr^\c_\phi(\bp) =
\{ A \in \Pyr_\phi^\leq \mid A \text{ is
justified row equivalent to column strict}\},
\]
however in type D we used a slightly modified definition of
justified row equivalent to
column strict.
In type D we say an s-frame $A$
is justified row equivalent to column strict if
it is justified row equivalent to column strict in the previous sense,
or if
the row equivalence class of the left justification of $A$
contains an element $B$
which is column strict everywhere,
except
the middle 2 boxes of
one column of $B$ contain 0.

\begin{Theorem} \label{T:main}
Let $\mf{g} = \mf{sp}_{2n}$ or $\mf{g} =\mf{so}_{2n}$.  Let $\bp$ be
an even multiplicity partition of $2n$, let $e \in \mf{g}$ be the
nilpotent element defined from the symmetric pyramid of $\bp$, and let
$A \in \Pyr^\leq_\phi(\bp)$.  Then the
$U(\g,e)$-module $L(A)$ is finite dimensional if and only if $A$ is
$\tilde C$-conjugate
an element of $\Pyr^\c_\phi(\bp)$.
\end{Theorem}

\begin{proof}
We associate $\lambda_A$ to $A$ as in \S\ref{ss:hwnotn}. First
observe that \cite[Lemma 2.4]{Jo2} reduces the case that $\lambda_A$
is non-regular to the case that $\lambda_A$ is regular, so we assume
that $\lambda_A$ is regular.
Let $\bp = (p_1^2,p_2^2,\dots,p_r^2)$,
where $p_i \ge p_{i+1}$ for all $i$.  Let $p_{i_1},\dots,p_{i_d}$ be
the minimal distinct parts of $\bp$ that are odd (respectively even) when $\g = \so_{2n}$
(respectively $\sp_{2n}$).
We write $c_k$ for the component group action corresponding
to $p_{i_k}$.

Suppose $A$ is justified row equivalent to column strict.
In all cases except the type  D case where $\bp$ is very even,
$L(A)$ is
finite dimensional by Theorem \ref{recs} and Corollaries~\ref{C:BGKconj} and \ref{C:BV2}.
In the type D case where $\bp$ is very even,
by \cite[Theorem 1.2.2 (v)]{Lo1}
we have that $\overline {G.e} \subseteq \mathcal{VA}(\Ann L(\lambda_A))$.
Also
Theorem \ref{recs} and Corollary \ref{C:BV2} give that
$\mathcal{VA}(\Ann L(\lambda_A))$ is the closure of one of the two
nilpotent $G$-orbits
corresponding to the partition $\bp$,
hence we have that
$\overline {G.e} = \mathcal{VA}(\Ann L(\lambda_A))$,
so by Corollary \ref{C:BGKconj} we have that $L(A)$ is finite dimensional.
So in all cases for any $w \in \tilde C$ we have that $L(w \cdot A)$ is finite dimensional by
Lemma~\ref{L:fdpreserved}.

To prove the converse we assume that $L(A)$ is finite dimensional.
Let $2r$ be the number of rows in
$A$. We proceed by induction on $r$. In the case that $r = 1$,
suppose the $A$ has row length $l$. Then by using
Corollary~\ref{C:BV2} one checks that $L(A)$ is finite dimensional
if and only if $\part(\RS(A)) = (l,l)$, or
$\part(\RS(A))
=(l+1,l-1)$ and $l$ is even and $\mf{g} = \mf{sp}_{2n}$,
or
$\part(\RS(A))
=(l+1,l-1)$ and $l$ is odd and $\mf{g} = \mf{so}_{2n}$.
In the former case there is nothing to prove: $A$ is justified row equivalent
to column strict by Theorem~\ref{recs}. In both
of the other
cases by Lemma~\ref{L:sharpdef} we have that $\part(\RS(c_1
\cdot A)) = (l,l)$, so again by Theorem~\ref{recs} $c_1
\cdot A$ is justified row equivalent to column strict.

Now suppose $r > 1$ and let $\r = \left\lan \sum_{i=n-p_r+1}^{n}
f_{i,i} \right\ran$. Then $\mf{g}^\r \iso \mf{gl}_{p_r} \oplus
\mf{g}^\phi_{2n-2p_r}$. Since $L^\r(\Lambda_A)$ is a finite
dimensional $U(\mf{g}^\r,e)$-module by Corollary~\ref{C:inductsl},
we get that $L(A^{1-r}_{r-1})$ is a finite dimensional $U(\mf{g}',
e')$-module, where $\mf{g}' \iso \mf{g}^\phi_{2n-2p_r}$ and $e' \in
\g'$ is the nilpotent element of $\g'$ defined from the symmetric
pyramid of $\bp' = (p_1^2,\dots,p_{r-1}^2)$. So by induction we can
apply some word $w$ in the elements of the set $\{c_1, \dots,
c_{d'}\}$, where $d' = d-1$ if $i_d = r$ and $d' = d$ is $i_d < r$
to $A$ to yield an s-frame $B$ which is justified row equivalent to
column strict. By replacing $A$ by $B$, and using
Theorem~\ref{recs}, we can assume that $A_{r-1}^{1-r}$ is justified
row equivalent to column strict.

Using Lemma~\ref{L:siwelldef} and Theorem~\ref{recs}, we see that if
one adjusts $A$ so that the middle $2r-2$ rows are left justified,
row $-r$ is left justified with row
$-r+1$, and row $r$ is right justified with
row $r-1$, then the resulting diagram is row equivalent to column
strict. Note that if $p_r = p_{r-1}$, then this implies that $A$
itself is justified row equivalent to column strict, so we assume
that this is not the case.

Let $\bq = \part(\RS(A))$ .
Now the above discussion and Lemma~\ref{altcRS} show that
\[
\bq^T \ge ((2r-1)^{2 p_r}, (2r-2)^{p_{r-1} - 2p_r},
        (2r-4)^{p_{r-2} - p_{r-1}}, \dots, 2^{p_1 - p_2})
\]
in the case that $p_r \le p_{r-1}/2$ and that
\[
\mathbf{q}^T \ge ((2r)^{2 p_r - p_{r-1}}, (2r-1)^{2(p_{r-1} -p_r)},
        (2r-4)^{p_{r-2} - p_{r-1}}, \dots, 2^{p_1 - p_2})
\]
in the case that $p_r > p_{r-1}/2$. This implies that
\[
  \mathbf{q} \le (p_1^2, \dots, p_{r-1}^2, 2 p_r)
\]
in the case that $p_r \le p_{r-1}/2$ and that
\[
\mathbf{q} \le (p_1^2, \dots, p_{r-1}^2, p_{r-1}, 2 p_r - p_{r-1})
\]
in the case that $p_r > p_{r-1}/2$.
Since $A$ has increasing rows we also have that
\[
\bq \ge (p_1^2, \dots, p_r^2).
\]
The content of $\bq$ is defined at the end of  \S\ref{ss:BV}.  From
Corollaries~\ref{C:BGKconj} and \ref{C:BV2} it follows that the
length of the content of $\bq$ is the same as the length of the content of $\bp$.

Suppose for this paragraph that $\g$ is of type C.  The content of
$\bp$ has length $2r+1$, which implies that
\[
\bq =(p_1^2, \dots, p_{r-1}^2, 2 p_r-a,a)
\]
for some $a$ where $1 \le a \le p_r$ in the case that $p_r \le
p_{r-1}/2$ and that
\[
\bq = (p_1^2, \dots, p_{r-1}^2, p_{r-1}-a, 2 p_r - p_{r-1}+a)
\]
for some $a$ where $0 \le a \le p_{r-1} - p_r$ in the case that $p_r
> p_{r-1}/2$.  The 3 first entries in
the content of $\bp$ are
\[
  \left(0, \frac{p_r+1}{2}, \frac{p_r+1}{2}\right)
\]
when $p_r$ is odd and
\[
  \left(0, \frac{p_r}{2}, \frac{p_r}{2}+1\right)
\]
when $p_r$ is even.  This implies that
\[
  a =
  \begin{cases}
    p_r & \text{ if $p_r \le p_{r-1}/2$, $p_r$ is odd,
      and $a$ is odd;} \\
    p_r -1 & \text{ if $p_r \le p_{r-1}/2$, $p_r$ is even,
      and $a$ is odd;} \\
    p_r & \text{ if $p_r \le p_{r-1}/2$, $p_r$ is even,
      and $a$ is even;} \\
    p_{r-1} - p_r & \text{ if $p_r > p_{r-1}/2$,
         $p_r$ is odd, and $p_{r-1} - a$ is odd;} \\
    p_{r-1} - p_r - 1  & \text{ if $p_r > p_{r-1}/2$,
         $p_r$ is even, and $p_{r-1} - a$ is odd;} \\
    p_{r-1} - p_r & \text{ if $p_r > p_{r-1}/2$,
         $p_r$ is even, and $p_{r-1} - a$ is even.}
  \end{cases}
\]
It also implies that no such $a$ exists when
       $p_r \le p_{r-1}/2$, $p_r$ is odd, and $a$ is even;
   or when  $p_r > p_{r-1}/2$, $p_r$ is odd, and
      $p_{r-1} -a$ is even.
Putting this all together shows that in all cases
\begin{equation} \label{e:case1}
\bq = (p_1^2, \dots, p_{r-1}^2, p_r^2),
\end{equation}
or
\begin{equation} \label{e:case2}
\bq = (p_1^2, \dots, p_{r-1}^2, p_r+1, p_r-1).
\end{equation}
Moreover, \eqref{e:case2} is only possible if $p_r$ is even.

If $\g$ is of type D, then similar arguments show that $\bq$ must be
as in \eqref{e:case1} or \eqref{e:case2}, and \eqref{e:case2} can
only occur if $p_r$ is odd.

If \eqref{e:case1} holds, then $A$ is justified row equivalent to
column strict by Theorem~\ref{recs}. So for the rest of the proof we
assume that \eqref{e:case2} holds.  We also assumed above that $p_r
< p_{r-1}$, which implies that $i_d = r$, so we can apply $c_d$ to
$A$.

For the arguments below it is useful to note that
\begin{equation} \label{ptranspose}
    \mathbf{q}^T = ((2r)^{p_r - 1}, (2r-1)^2, (2r-2)^{p_{r-1} - p_r -1},
      (2r-4)^{p_{r-2}-p_{r-1}}, \dots, 2^{p_1 - p_2}).
\end{equation}

Let $A' = \bar s_{1} \bar s_{2} \dots \bar s_{r-1} A$, so that the
shortest rows of $A'$ are the middle two, and they have length
$p_r$. Also let $F'$ be the frame of $A'$.

Let $\br = \part(\RS(A'^{-1}_{1}))$.  We aim to show that $\br =
(p_r+1, p_r-1)$.

To do this first note that $\br \ge (p_r, p_r)$. If the first part
of $\br$ were larger than $p_r+1$, then since the rows of $A'$ are
increasing, we would have
\[
\bq > (p_1^2, \dots, p_{r-1}^2, p_r+1, p_r-1),
\]
which contradicts \eqref{e:case2}.

Now suppose that $\br = (p_r^2)$. Using Lemma~\ref{L:siwelldef} and
Theorem~\ref{recs} we see that $A'^+$ is justified row equivalent to
column strict, which implies that we can find $p_r$ disjoint
descending chains in $\word(A')$ of length $2r$. This implies that
$\bq^T \ge ((2r)^{p_r}, 1^{2n-2 r p_r})$ by Lemma~\ref{altRS}, which
contradicts \eqref{ptranspose}.

Now the arguments in the
previous two paragraphs prove that $\br = (p_r+1, p_r-1)$ as desired.
Let $a_1 < a_2 < \dots < a_{p_r}$ be the
entries in row $-1$ of $A_1'^{-1}$. By Lemma~\ref{L:sharpdef} we have
that the $\sharp$-element of $(a_1, \dots, a_{p_r})$ is defined, and we
let $a_j$ be the $\sharp$-element.  Thus $c \cdot A'$ is defined,
and by Lemma~\ref{L:sharpdef} we have that $\part(\RS(c \cdot
A'^{-1}_1 )) = (p_r,p_r)$.

Let $\mathbf{q}^\sharp = \part(\RS(c \cdot A'))$. Note that if
$u,w$ are words of integers and $a, b \in \Z$ with $a<b$, then
$\RS(u a b w) \ge \RS(u b a w)$ by Lemma~\ref{altRS}, because every
collection of disjoint increasing sequences in $u b a w$ is a
collection of disjoint increasing sequences in $u a b w$. Now
Remark~\ref{R:sharp} says that precisely 2 such swaps are required
to get from something Knuth-equivalent to $\word(A')$ to something
Knuth-equivalent to $\word(c \cdot A')$.  Therefore, we have $\bq^\sharp
\le \bq$.

Now we have that
\[
\bq^\sharp = (p_1^2, \dots, p_{r-1}^2, p_r+1, p_r-1)
\]
or
\[
\bq^\sharp = (p_1^2, \dots, p_{r-1}^2, p_r, p_r).
\]

>From Remark~\ref{R:altRS} and \eqref{ptranspose}, we see that we can
find $p_r-1$ disjoint descending chains in the $\word(A')$ of length
$2r$, and we can find $2$ other decreasing chains of length $2r-1$
which are disjoint from the chains of length $2 r$. These chains
must include every number which occurs in $A'^{-1}_{1}$. Furthermore
it is possible to adjust them so that one of the chains of length
$2r-1$ contains $a_j$, and other chain of length $2r -1$ contains
$-a_j$. To do this explicitly note that the existence of the chains
of length $2r$ and $2r-1$ implies that we can form $p_r$ disjoint
descending chains of length $r$ which end in $a_1, \dots, a_{p_r}$,
as well as $p_r$ disjoint descending chains of length $r$ which
start in $-a_{p_r}, \dots, -a_1$. Since the $\sharp$-element is
defined,  we can assign to each $i  \in \{1, \dots, p_r\}\setminus
\{j\}$ some $i' \in \{1, \dots, p_r\}\setminus \{j\}$  such that
$a_i > -a_{i'}$.  To be clear, this assignment can be made so that
$i'_1 = i'_2$ implies that $i_1 = i_2$. Also note that in the type D
case our zero convention from Algorithm \ref{A:BV2} does not affect
things here since in this case $p_r$ is odd, so by Lemma
\ref{L:sharp1} we can make this assignment so that if $a_i = 0$,
then $a_{i'} \neq 0$. Thus we can join $p_r-1$ of these chains
together at $a_i, -a_{i'}$ for $i \in \{1, \dots, p_r\} \setminus
\{j\}$, to get $p_r -1$ disjoint descending chains of length $2 r$,
and still have 2 other disjoint descending chains of length $r$, one
of which ends in $a_j$, and one of which starts in $-a_j$.

Now after $c$ is applied to $A'^{-1}_1$ we have that $a_j$
now occurs in row
$1$, while $-a_j$ now occurs in row $-1$. All of the descending chains
created above still exist, and now we can join the last 2 chains
containing $-a_j$ and $a_j$ to form one more descending chain of
length $2 r$. So $(\bq^\sharp)^T$ is larger than or equal to the
partition $((2r)^{p_r},1^{2n-2rp_r})$, which implies that
$\bq^\sharp = (p_1^2, \dots, p_r^2)= \bp$.

Using Lemma~\ref{swaprs}, we see that $\RS(c_d \cdot A) = \bp$.
Hence, by Theorem \ref{recs} it is justified row equivalent to
column strict.
\end{proof}

Theorem \ref{T:main} and Lemma \ref{L:srecs} immediately give the following
corollary:
\begin{Corollary}
    Let $\mf{g}$, $\bp$, $A$, and $e$ be as in Theorem \ref{T:main},
    and also suppose that all the
parts of $\bp$ have the same parity.
Then the
$U(\g,e)$-module $L(A)$ is finite dimensional if and only if $A$ is
$\tilde C$-conjugate
to an s-table
which is row equivalent to a column strict s-table.
\end{Corollary}

Next we give a corollary of Theorem \ref{T:main} saying that all finite dimensional
irreducible $U(\g,e)$-modules with integral central character can be obtained by restricting
certain $U(\h)$-modules when all parts of $\p$ have the same parity.  In this case $\k = 0$, so $\tilde \p
= \p$.  The {\em Miura map} is by definition
the composition of the inclusion $U(\g,e) \into U(\p)$ with the surjection $U(\p) \onto U(\h)$.
It is known that the Miura map is injective, see \cite[Remark 2.2]{Pr2}, which allows us to restrict
$U(\h)$-modules to $U(\g,e)$.

\begin{Corollary} \label{C:evencase}
    Let $\mf{g}$, $\bp$, and $e$ be as in Theorem \ref{T:main},
    and also suppose that all the
parts of $\bp$ have the same parity.
Let $L$ be a finite dimensional irreducible $U(\g,e)$-module with integral
central character.  Then there exists a finite dimensional $U(\h)$-module $M$ such that $L$
is a subquotient of the restriction of $M$ to $U(\g,e)$.
\end{Corollary}

\begin{proof}
First let $A \in \Pyr^\c_\phi(\bp)$,
so by Lemma \ref{L:srecs}
there exists $B \in \bar A$ which is column strict as an s-table.

We claim that the weight $\lambda_B \in \t^*$ is dominant for $\h$
with respect to the Borel subalgebra $\b^h$ of $\h = \g^h$. To see
this, note that
\[
\tilde \rho + \gamma = \frac{1}{2} \left(
\sum_{\substack{\alpha \in \Phi_+ \\ \g_{\alpha} \in \h}}
\alpha
+
\sum_{i=1}^r \beta_i
\right).
\]
Note that the first sum is the usual ``choice of  $\rho$'' for $\h$,
and the second sum is orthogonal to every root occurring in the
first sum, so $\tilde \rho + \gamma$ is a ``choice of $\rho$'' for
$\h$. So the highest weight $U(\h)$-module with highest weight
$\lambda_B - (\tilde \rho+ \gamma)$ with respect to the Borel
subalgebra $\b^h$ is finite dimensional. We denote this module by $M
= L_\h(\lambda_B)$ and let $v_+ \in M$ be a highest weight vector.
We can restrict $M$ to a $U(\g,e)$-module through the Miura map.  It
is clear that $v_+ \in M$ lies in a maximal $\t^e$-weight space
$M_\mu$, where $\mu \in (\t^e)^*$. Thus $M_\mu$ can be viewed as a
$U(\g_0,e)$-module as in \S\ref{ss:recaphw}.  As such it is clear
that $v_+$ spans a one dimensional $U(\g_0,e)$-submodule. Through
the isomorphism $\xi_{\tilde \rho} : U(\g_0,e) \to S(\t)^{W_0}$ from
\eqref{e:xi}, this identifies with the $S(\t)^{W_0}$-module
$V_{\Lambda_A}$.  It follows that the $U(\g,e)$-submodule of $M$
generated by $v_+$ is a highest weight $U(\g,e)$-module of type
$\Lambda_A$, and thus has a quotient isomorphic to $L(A)$.

Now let $A \in \Pyr^\leq_\phi(\bp)$ such that $L(A)$ is finite
dimensional.  Then by Theorem~\ref{T:main}, there exists $c \in \tilde C$
such that $c \cdot A$ is row equivalent to column strict. Therefore,
$L(c \cdot A)$ is isomorphic to a subquotient of a finite
dimensional $U(\h)$-module $M$ by the previous paragraph.  We write
$c \cdot M$ for the $U(\h)$-module obtained by twisting $M$ by $c$,
which is defined similarly to the case of $U(\g,e)$-modules in
\S\ref{ss:hwcomp}.
Now using the result in \cite{BroG} that $L(A) \iso c \cdot L(c
\cdot A)$, we see that $L(A)$ is isomorphic to a subquotient of $c
\cdot M$.  (We note that $c \cdot M$ is isomorphic to $M$ if $c \in
C$.)
\end{proof}

\begin{Remark} \label{R:wrong}
    The hypothesis that $A$ in Theorem \ref{T:main} has entries
    all from $\Z$, or all from
$\Z + \frac{1}{2}$, ensures that $\lambda_A \in \t^*_\Z$.  In the example below we
demonstrate that this hypothesis is necessary.  As usual for complex numbers $x,y$  we
say $x < y$ if $y - x \in \Z_{>0}$.
If we allow tables to have entries to have entries from $\C$,
then we say that a (justified) table is column strict provided its columns are decreasing with respect
to this partial order.
Now if $\mf{g} = \mf{sp}_4$ and $e \in \g$ has Jordan type $(2^2)$,
then by \cite[Theorem 1.2]{Bro2} the $U(\mf{g},e)$-module $L(A)$, where
\[
A=
        \begin{array}{c}
\begin{picture}(40,40)
\put(0,0){\line(1,0){40}}
\put(0,20){\line(1,0){40}} \put(0,40){\line(1,0){40}}
\put(0,0){\line(0,1){20}} \put(20,0){\line(0,1){20}}
\put(40,0){\line(0,1){20}} \put(0,20){\line(0,1){20}}
\put(20,20){\line(0,1){20}} \put(40,20){\line(0,1){20}}
\put(10,10){\makebox(0,0){{$-1$}}} \put(30,10){\makebox(0,0){{$-\pi$}}}
\put(10,30){\makebox(0,0){{$\pi$}}} \put(30,30){\makebox(0,0){{1}}}
\put(20,20){\circle*{3}}
\end{picture}
\end{array},
\]
is finite dimensional, however $A$ is not $\tilde C$-conjugate
to a row equivalent to column strict s-table.

\end{Remark}

The following corollary is immediate from Theorem~\ref{T:main} and
the map $\cdot ^\dagger$ from \eqref{eq:dagger}.
Recall that in the type D case
$C \subset \tilde C$ is a subgroup of index 2.
With this in mind,
we use
$\Pyr_{+,s}^\c(\bp)$
to denote
the elements $A$ of $\Pyr_+^\c(\bp)$
for which there exists $c \in \tilde C \setminus C$
such that $c \cdot A = A$.
We also note that
if $A$ has no boxes filled with $0$ then
$A \notin \Pyr_{+,s}^\c(\bp)$.

\begin{Corollary} \label{C:prim}
Let $\g = \sp_{2n}$ or $\so_{2n}$, and let $e$ be an even
multiplicity nilpotent element of $\g$.
If $\g = \sp_{2n}$, then
\[
\{\Ann_{U(\g)}L(\lambda_A) \mid A \in \Pyr_-^\c(\bp) \}
\]
is a complete set of
pairwise distinct
primitive ideals of $U(\g)$ with integral central character
and associated variety $\bar{G \cdot e}$.

If $\g = \so_{2n}$, then
\begin{align*}
    \{\Ann_{U(\g)} & L(\lambda_A)   \mid A \in \Pyr_{+,s}^\c(\bp)\} \\
&\cup
\{\Ann_{U(\g)}L(\lambda_A)  \mid A \in \Pyr_+^\c(\bp) \setminus  \Pyr_{+,s}^\c(\bp) \} \\
&\cup
\{\Ann_{U(\g)}L(\lambda_{c_1 \cdot A}  \mid A \in \Pyr_+^\c(\bp) \setminus  \Pyr_{+,s}^\c(\bp) \}
\end{align*}
is a complete set of
pairwise distinct
primitive ideals of $U(\g)$ with integral central character
and associated variety $\bar{G \cdot e}$.
\end{Corollary}


\begin{thebibliography}{CPS33}

\bibitem[BB]{BB}
J. Brown and J. Brundan,
{\em Elementary invariants for centralizers of nilpotent matrices},
 J. Austral. Math. Soc. {\bf 86} (2009), 1--15,
{\tt math.RA/0611024}.

\bibitem[Bro1]{Bro1} J.~Brown,
{\em Twisted Yangians and finite W-algebras},  Transform.\ Groups
{\bf 14} (2009),  no.\ 1, 87--114.

\bibitem[Bro2]{Bro2} \bysame,
{\em Representation theory of rectangular finite $W$-algebras},
preprint, arXiv:1003.2179v1 (2010).

\bibitem[BroG]{BroG} J.~Brown and S.~M.~Goodwin,
{\em Changing the highest weight theory for finite $W$-algebras},
in preparation (2010).

\bibitem[Bru1]{Bru1} J.~Brundan, {\em Centers of degenerate
cyclotomic Hecke algebras and parabolic category $\cO$}, Represent.\
Theory {\bf 12} (2008), 236--259.

\bibitem[Bru2]{Bru2} \bysame,
{\em Symmetric functions, parabolic category $\cO$, and the Springer
fiber}, Duke Math.\ J.\ {\bf 143} (2008), 41--79.

\bibitem[BruG]{BruG}
J.~Brundan and S.~M.~Goodwin, {\em Good grading polytopes}, Proc.\
London Math.\ Soc.\ {\bf 94} (2007), 155--180.

\bibitem[BGK]{BGK} J.~Brundan, S.~M.~Goodwin and A.~Kleshchev,
{\em Highest weight theory for finite $W$-algebras}, Internat.\
Math.\ Res.\ Notices, {\bf 15} (2008), Art.\ ID rnn051.

\bibitem[BK1]{BK1}
J.~Brundan and A.~Kleshchev, {\em Shifted Yangians and finite
$W$-algebras}, Adv.\ Math.\ {\bf  200} (2006), 136--195.

\bibitem[BK2]{BK2}
\bysame, 
{\em Representations of shifted Yangians and finite $W$-algebras},
Mem.\ Amer.\ Math.\ Soc. {\bf 196} (2008).

\bibitem[BK3]{BK3} \bysame,
{\em Schur-Weyl duality for higher levels},  Selecta Math {\bf 14}
(2008), 1--57 (2008).

\bibitem[BV]{BV}
D.~Barbasch and D.~Vogan, {\em Primitive ideals and orbital
integrals in complex classical groups}, Math.\ Ann.\ {\bf 259}
(1982), 153--199.

\bibitem[D${}^3$HK]{DDDHK}
A.~D'Andrea, C.~De Concini, A. De Sole, R. Heluani and V. Kac, {\em
Three equivalent definitions of finite $W$-algebras}, appendix to
\cite{DK}.

\bibitem[DK]{DK}
A.~De Sole and V.~Kac,  {\em Finite vs affine $W$-algebras}, Jpn.\
J.\ Math.\ {\bf 1} (2006), 137--261.

\bibitem[Du]{Du}
M.~Duflo, {\em Sur la classification des id\'eaux primitif dans
l'algebre enveloppante d'une algebre de Lie semisimple}, Ann.\ of
Math.\ {\bf 105} (1977), 107--120.

\bibitem[EK]{EK}
A.~Elashvili and V.~Kac,  {\em Classification of good gradings of
simple Lie algebras}, in Lie groups and invariant theory
(E.~B.~Vinberg ed.), pp. 85--104, Amer.\ Math.\ Soc.\ Transl.\ {\bf
13}, AMS, 2005.

\bibitem[F]{F}
W.~Fulton,
{\em Young tableaux},
London Math.\ Soc.\ Stud.\ Texts {\bf 35}, Cambridge University Press, Cambridge, UK, 1997..

\bibitem[GG]{GG}
W.~L.~Gan and V.~Ginzburg, {\em Quantization of Slodowy slices},
Internat. Math. Res. Notices {\bf 5} (2002), 243--255.

\bibitem[Gi]{Gi}
V.~Ginzburg, {\em Harish-Chandra bimodules for quantized Slodowy
slices}, Represent.\ Theory {\bf 13} (2009), 236--271.

\bibitem[Gr]{Gr}
C. Green, {\em Some partitions associated with a partially ordered set},
J.\ of Combin.\ Theory, Ser.\ A {\bf 20} (1976), 69--79.

\bibitem[Ja]{Ja}
J. C. Jantzen, {\em Nilpotent orbits in representation theory},
Progress in Math., vol. 228, Birkh\"auser, 2004.

\bibitem[Jo1]{Jo1}
A.~Joseph,
{\em Towards the Jantzen conjecture III},
Compositio Math. {\bf 41} (1981), 23--30.

\bibitem[Jo2]{Jo2} \bysame,
{\em On the associated variety of a primitive ideal}, J.\
Algebra {\bf 93} (1985), 509--523.

\bibitem[Ko]{Ko}
B.~Kostant, {\em On Whittaker modules and representation theory},
Invent.\ Math.\ {\bf 48} (1978), 101--184.

\bibitem[Lo1]{Lo1}
I.~Losev, {\em Quantized symplectic actions and $W$-algebras}, J.\
Amer.\ Math.\ Soc.\  {\bf 23}  (2010),  no.\ 1, 35--59.

\bibitem[Lo2]{Lo2} \bysame,
{\em Finite dimensional representations of W-algebras}, preprint,
arXiv:0807.1023 (2008).

\bibitem[Lo3]{Lo3} \bysame,
{\em On the structure of the category $\mathcal O$ for
$W$-algebras}, preprint, arXiv:0812.1584 (2008).

\bibitem[Lo4]{Lo4} \bysame,
{\em Finite $W$-algebras}, preprint, arXiv:1003.5811 (2010).

\bibitem[Lu]{Lu} G.~Lusztig, {\em A class of irreducible
representations of a Weyl group}, Proc.\ Nedeerl.\ Akad., Series A
{\bf 82} (1979), 323--335.

\bibitem[MS]{MS}
D. Mili\v c\'ic and W. Soergel,
{\em  The composition series of modules induced from Whittaker modules},
Comment.\ Math.\ Helv.\ {\bf 72} (1997), 503--520.

\bibitem[Pr1]{Pr1}
A.~Premet, {\em Special transverse slices and their enveloping
algebras}, Adv.\ in Math.\ {\bf 170} (2002), 1--55.

\bibitem[Pr2]{Pr2} \bysame,
{\em Enveloping algebras of Slodowy slices and the Joseph ideal},
J.\ Eur.\ Math.\ Soc.\ {\bf 9} (2007), 487--543.

\bibitem[Pr3]{Pr3} \bysame,
{\em Primitive ideals, non-restricted representations and finite
$W$-algebras},
Mosc.\ Math.\ J. {\bf 7} (2007), 743--762.

\bibitem[Pr4]{Pr4} \bysame,
{\em Commutative quotients of finite W-algebras}, preprint,
arXiv:0809.0663 (2008).

\bibitem[Sk]{Sk} S.~Skryabin, {\em A category equivalence},
appendix to \cite{Pr1}.

\end{thebibliography}
\end{document}